\documentclass[11pt]{article}
\usepackage{latexsym,amsmath,amssymb,amsfonts,amsthm,bbm}
\usepackage{natbib,natbibspacing}
\usepackage{dsfont}
\usepackage{enumerate}
\usepackage{graphicx}
\usepackage{graphics}
\usepackage{psfrag}
\usepackage{caption}

\def\bw{\mathbf w}
\def\bx{\mathbf x}

\def\bz{\mathbf z}

\def\bX{\mathbf X}

\def\bZ{\mathbf Z}
\def\ba{\mathbf a}
\def\bb{\mathbf b}
\def\bA{\mathbf A}
\def\bB{\mathbf B}
\def\bL{\mathbf L}
\def\balpha{\boldsymbol \alpha}
\def\bbeta{\boldsymbol \beta}

\newtheorem{example}{Example}
\newtheorem{thm}{Theorem}
\newtheorem{prop}[thm]{Proposition}
\newtheorem{lem}[thm]{Lemma}
\newtheorem{cor}[thm]{Corollary}

\newcommand{\abovebelow}[2]{\genfrac{}{}{0pt}{}{#1}{#2}}

\DeclareMathOperator*{\argmax}{argmax}
\DeclareMathOperator*{\dom}{dom}

\newenvironment{prooftitle}[1]{{\noindent \textsc{Proof #1}}\\}

\begin{document}

\title{Generalised additive and index models \\with shape constraints}

\author{Yining Chen and Richard J. Samworth\\ 
  Statistical Laboratory\\
  University of Cambridge\\
  \{y.chen, r.samworth\}@statslab.cam.ac.uk
}

\maketitle
\begin{abstract}
We study generalised additive models, with shape restrictions (e.g. monotonicity, convexity, concavity) imposed on each component of the additive prediction function.  We show that this framework facilitates a nonparametric estimator of each additive component, obtained by maximising the likelihood.  The procedure is free of tuning parameters and under mild conditions is proved to be uniformly consistent on compact intervals.  More generally, our methodology can be applied to generalised additive index models.  Here again, the procedure can be justified on theoretical grounds and, like the original algorithm, possesses highly competitive finite-sample performance.  Practical utility is illustrated through the use of these methods in the analysis of two real datasets.  Our algorithms are publicly available in the \texttt{R} package \texttt{scar}, short for \textbf{s}hape-\textbf{c}onstrained \textbf{a}dditive \textbf{r}egression.
\end{abstract}


\section{Introduction}
\label{Sec:intro}
Generalised additive models (GAMs) \citep{HastieTibshirani1986, HastieTibshirani1990, Wood2006} have become an extremely popular tool for modelling multivariate data. They are designed to enjoy the flexibility of nonparametric modelling while avoiding the curse of dimensionality \citep{Stone1986}.  Mathematically, suppose that we observe pairs $(\bX_1,Y_1),\ldots, (\bX_n,Y_n)$, where $\bX_i = (X_{i1},\ldots,X_{id})^T \in \mathbb{R}^d$ is the predictor and $Y_i \in \mathbb{R}$ is the response, for $i = 1,\ldots,n$.  A generalised additive model relates the predictor and the mean response $\mu_i = \mathbb{E}(Y_i)$ through
\[
	g(\mu_i) = f(\bX_i) = \sum_{j=1}^d f_j(X_{ij}) + c,
\]
where $g$ is a specified link function, and where the response $Y_i$ follows an exponential family distribution.  Here $c\in\mathbb{R}$ is the intercept term and for every $j = 1,\ldots,d$, the additive component function $f_j:\mathbb{R} \rightarrow \mathbb{R}$ is assumed to satisfy the identifiability constraint $f_j(0) = 0$.  Our aim is to estimate the additive components $f_1,\ldots,f_d$ together with the intercept $c$ based on the given observations.  Standard estimators are based on penalised spline-based methods \citep[e.g.][]{Wood2004,Wood2008}, and involve tuning parameters whose selection is not always straightforward, especially if different additive components have different levels of smoothness, or if individual components have non-homogeneous smoothness.

In this paper, we propose a new approach, motivated by the fact that the additive components of $f$ often follow certain common shape constraints such as monotonicity or convexity.  The full list of constraints we consider is given in Table~\ref{Tab:shapeclass}, with each assigned a numerical label to aid our exposition.  By assuming that each of $f_1,\ldots,f_d$ satisfies one of these nine shape restrictions, we show in Section~\ref{Sec:gam} that it is possible to derive a nonparametric maximum likelihood estimator, which requires no choice of tuning parameters and which can be computed using fast convex optimisation techniques.  In Theorem~\ref{Thm:SCMLEconsistency}, we prove that under mild regularity conditions, it is uniformly consistent on compact intervals.

\begin{table}[!ht]
  \centering
  \begin{tabular}{ c c | c c | c c } \hline
  shape constraint & label & shape constraint & label & shape constraint & label \\\hline
  linear  & 1 & monotone increasing & 2 & monotone decreasing & 3 \\
  convex  & 4 & convex increasing   & 5 & convex decreasing   & 6 \\
  concave & 7 & concave increasing  & 8 & concave decreasing  & 9 \\\hline
  \end{tabular}
  \caption{Different shape constraints and their corresponding labels}
  \label{Tab:shapeclass}
\end{table}

More generally, as we describe in Section~\ref{Sec:gaim}, our approach can be applied to generalised additive index models (GAIMs), in which the predictor and the response are related through
\begin{align}
\label{Eq:gaim}
	g(\mu_i) = f^I(\bX_i) = f_1(\balpha_1^T \bX_i) + \ldots +  f_{m}(\balpha_{m}^T \bX_i) + c, 
\end{align}
where the value of $m \in \mathbb{N}$ is assumed known, where $g$ is a known link function, and where the response $Y_i$ again follows an exponential family distribution. Here, $\balpha_1,\ldots,\balpha_{m} \in \mathbb{R}^d$ are called the \emph{projection indices}, $f_1,\ldots,f_{m}:\mathbb{R} \rightarrow \mathbb{R}$ are called the \emph{ridge functions} (or sometimes, \emph{additive components}) of $f^I$, and $c \in \mathbb{R}$ is the intercept.  Such index models have also been widely applied, especially in the area of econometrics \citep{LiRacine2007}. When $g$ is the identity function, the model is also known as projection pursuit regression \citep{FriedmanStuetzle1981}; when $m=1$, the model reduces to the single index model \citep{Ichimura1993}. By imposing shape restrictions on each of $f_1,\ldots,f_m$, we extend our methodology and theory to this setting, allowing us to estimate simultaneously the projection indices, the ridge functions and the intercept.

The challenge of computing our estimators is taken up in Section~\ref{Sec:compute}, where our algorithms are described in detail.  In Section~\ref{Sec:sim}, we summarise the results of a thorough simulation study designed to compare the finite-sample properties of \texttt{scar} with several alternative procedures.  We conclude in Section~\ref{Sec:real} with two applications of our methodology to real datasets concerning doctoral publications in biochemistry and the decathlon.  The proofs of our main results can be found in the Appendix; various auxiliary results are given in the online supplementary material.

This paper contributes to the larger literature of regression in the presence of shape constraints.  In the univariate case, and with the identity link function, the properties of shape-constrained least squares procedures are well-understood, especially for the problem of isotonic regression. See, for instance, \citet{Brunk1958}, \citet{Brunk1970} and \citet{BBBB1972}. For the problem of univariate convex regression, see \citet{HansonPledger1976}, \citet{GJW2001}, \citet{GJW2008} and \citet{GuntuboyinaSen2013}. These references cover consistency, local and global rates of convergence,  and computational aspects of the estimator.  \citet{MammenYu2007} studied additive isotonic regression with the identity link function.  During the preparation of this manuscript, we became aware of the work of \citet{Meyer2013a}, who developed similar methodology (but not theory) to ours in the Gaussian, non-index setting.  The problem of GAMs with shape restrictions was also recently studied by \citet{PyaWood2014}, who proposed a penalised spline method that is compared with ours in Section~\ref{Sec:sim}.  Finally, we mention that recent work by \citet{KimSamworth2014} has shown that shape-restricted inference without further assumptions can lead to slow rates of convergence in higher dimensions. The additive or index structure therefore becomes particularly attractive in conjunction with shape constraints as an attempt to evade the curse of dimensionality.

\section{Generalised additive models with shape constraints}
\label{Sec:gam}

\subsection{Maximum likelihood estimation}

Recall that the density function of a natural exponential family (EF) distribution with respect to a reference measure (either Lebesgue measure on $\mathbb{R}$ or counting measure on $\mathbb{Z}$) can be written in the form
\[
	f_Y(y;\mu,\phi) = h(y,\phi) \, \exp \left\{\frac{yg(\mu)-B(g(\mu))}{\phi} \right\},
\]
where $\mu \in \mathcal{M} \subseteq \mathbb{R}$ and $\phi \in \Phi \subseteq (0,\infty)$ are the mean and dispersion parameters respectively.  To simplify our discussion, we restrict our attention to the most commonly-used natural EF distributions, namely, the Gaussian, Gamma, Poisson and Binomial families, and take $g$ to be the canonical link function.  Expressions for $g$ and the (strictly convex) log-partition function $B$ for the different exponential families can be found in Table~\ref{Tab:glmformula}.  The corresponding distributions are denoted by $\mathrm{EF}_{g,B}(\mu,\phi)$, and we write $\dom(B) = \{\eta \in \mathbb{R}:B(\eta) < \infty\}$ for the domain of $B$.  As a convention, for the Binomial family, the response is scaled to take values in $\{0,1/T,2/T,\ldots,1\}$ for some known $T \in \mathbb{N}$. 
\begin{table}[!ht]
  \centering
  \begin{tabular}{ c c c c c c c} \hline
  exponential family & $g(\mu)$ & $B(\eta)$ & dom($B$) & $\mathcal{M}$ & $\Phi$ \\\hline
  Gaussian  & $\mu$ & $\eta^2/2$  & $\mathbb{R}$ & $\mathbb{R}$ & $(0,\infty)$ \\
  Gamma  & $-\mu^{-1}$ & $-\log(-\eta)$ & $(-\infty, 0)$ & $(0,\infty)$ & $(0,\infty)$ \\
  Poisson  & $\log\mu$ & $e^\eta$ & $\mathbb{R}$ & $(0,\infty)$ & $\{1\}$ \\
  Binomial & $\log\frac{\mu}{1-\mu}$ & $\log(1+e^\eta)$ & $\mathbb{R}$ & $(0,1)$ & $\{1/T\}$ \\\hline
  \end{tabular}
  \caption{Exponential family distributions, their corresponding canonical link functions, log-partition functions and mean and dispersion parameter spaces.}
  \label{Tab:glmformula}
\end{table}

If $(\bX_1,Y_1), \ldots, (\bX_n,Y_n)$ are independent and identically distributed pairs taking values in $\mathbb{R}^d \times \mathbb{R}$, with $Y_i | \bX_i \sim \mathrm{EF}_{g,B}\bigl(g^{-1}(f(\bX_i)),\phi\bigr)$ for some \emph{prediction function} $f:\mathbb{R}^d\rightarrow \dom (B)$, then the (conditional) log-likelihood of $f$ can be written as 
\[
	\frac{1}{\phi} \sum_{i=1}^n \{Y_i f(\bX_i) - B(f(\bX_i))\} + \sum_{i=1}^{n} \log h(Y_i,\phi).
\]
Since we are only interested in estimating $f$, it suffices to consider the \emph{scaled partial log-likelihood} 
\[
	\bar{\ell}_{n,d}(f) \equiv \bar{\ell}_{n,d}\big(f;(\bX_1,Y_1),\ldots,(\bX_n,Y_n) \big) := \frac{1}{n} \sum_{i=1}^n \left\{ Y_i f(\bX_i) - B(f(\bX_i)) \right\} \equiv \frac{1}{n} \sum_{i=1}^n \ell_{i,d}(f),
\]
say. In the rest of this section, and in the proof of Proposition~\ref{Prop:gamexistunique} and Theorem~\ref{Thm:SCMLEconsistency} in the appendix, we suppress the dependence of $\bar{\ell}_{n,d}(\cdot)$ and $\ell_{i,d}(\cdot)$ on $d$ in our notation.

Let $\bar{\mathbb{R}} = \mathbb{R} \cup \{-\infty,\infty\}$ denote the extended real line.  In order to guarantee the existence of our estimator, it turns out to be convenient to extend the definition of each $\ell_i$ (and therefore $\bar{\ell}_n$) to all $f: \mathbb{R}^d \rightarrow \bar{\mathbb{R}}$, which we do as follows:
\begin{enumerate}[1.]
\setlength{\itemsep}{0pt}
\setlength{\parskip}{0pt}
\setlength{\parsep}{0pt}
\item For the Gamma family, if $f(\bX_i) \ge 0$, then we take $\ell_{i}(f) = -\infty$.  This is because the log-partition function of the Gamma family has domain $(-\infty,0)$, so $f$ must be negative at $\bX_i$ in order for $\ell_i(f)$ to be well-defined.
\item If $f(\bX_i) = -\infty$, then we set $\ell_{i}(f) = \lim_{a \rightarrow -\infty} Y_i a - B(a)$. Similarly, if $f(\bX_i) = \infty$ (in the Gaussian, Poisson or Binomial setting), then we define $\ell_{i}(f) = \lim_{a \rightarrow \infty} Y_i a - B(a)$.  Note that both limits always exist in $\bar{\mathbb{R}}$.
\end{enumerate}
For any $\bL_d = (l_1,\ldots,l_d)^T \in \{1,2,\ldots,9\}^d$, let $\mathcal{F}^{\bL_d}$ denote the set of functions $f:\mathbb{R}^d \rightarrow \mathbb{R}$ of the form
\[
f(\bx) = \sum_{j=1}^d f_j(x_j) + c
\]
for $\bx = (x_1,\ldots,x_d)^T \in \mathbb{R}^d$, where for every $j = 1,\ldots,d$, $f_j:\mathbb{R} \rightarrow \mathbb{R}$ is a function obeying the shape restriction indicated by label $l_j$ and satisfying $f_j(0) = 0$, and where $c \in \mathbb{R}$.  Whenever $f$ has such a representation, we write $f \stackrel{\mathcal{F}^{\bL_d}}{\sim} (f_1,\ldots,f_d,c)$, and call $\bL_d$ the \emph{shape vector}.
The pointwise closure of $\mathcal{F}^{\bL_d}$ is defined as 
\[
\mathrm{cl}(\mathcal{F}^{\bL_d}) = \Big\{ f: \mathbb{R}^d \rightarrow \bar{\mathbb{R}} \Big| \exists f^1, f^2, \ldots \in \mathcal{F}^{\bL_d} \mbox{ s.t. } \lim_{k\rightarrow\infty}f^k(\bx)=f(\bx) \mbox{ for every } \bx \in \mathbb{R}^d \Big \}.
\]

For a specified shape vector $\bL_d$, we define the shape-constrained maximum likelihood estimator (SCMLE) as 
\[
	\hat{f}_n \in \argmax_{f \in \mathrm{cl}(\mathcal{F}^{\bL_d})} \, \bar{\ell}_{n}(f).
\]
Like other shape restricted regression estimators, $\hat{f}_n$ is not unique in general. However, as can be seen from the following proposition, the value of $\hat{f}_n$ is uniquely determined at $\bX_1,\ldots,\bX_n$.

\begin{prop}
\label{Prop:gamexistunique}
The set $\hat{S}_n^{\bL_d} = \argmax_{f \in \mathrm{cl}(\mathcal{F}^{\bL_d})} \bar{\ell}_n(f)$ is non-empty. Moreover, all elements of $\hat{S}_n^{\bL_d}$ agree at $\bX_1, \ldots, \bX_n$.
\end{prop}

\noindent\textbf{Remarks}: 
\begin{enumerate}[1.]
\setlength{\itemsep}{0pt}
\setlength{\parskip}{0pt}
\setlength{\parsep}{0pt}
\item As can be seen from the proof of Proposition~\ref{Prop:gamexistunique}, if the EF distribution is Gaussian or Gamma, then $\hat{S}_n^{\bL_d} \cap \mathcal{F}^{\bL_d} \neq \emptyset$.
\item Under the Poisson setting, if $Y_i = 0$, then it might happen that $\hat{f}_n(\bX_i) = -\infty$. Likewise, for the Binomial GAM, if $Y_i = 0 \mbox{ or } 1$, then it is possible to have $\hat{f}_n(\bX_i) = -\infty \mbox{ or } \infty$, respectively. This is why we maximise over the closure of $\mathcal{F}^{\bL_d}$ in our definition of SCMLE.
\end{enumerate}


\subsection{Consistency of the SCMLE}
In this subsection, we show the consistency of $\hat{f}_n$ in a random design setting. We will impose the following assumptions:
\begin{enumerate}
\setlength{\itemsep}{0pt}
\setlength{\parskip}{0pt}
\setlength{\parsep}{0pt}
\item[\textbf{(A.1)}] $(\bX,Y),(\bX_1,Y_1), (\bX_2,Y_2),\ldots$ is a sequence of independent and identically distributed pairs taking values in $\mathbb{R}^d \times \mathbb{R}$.  
\item[\textbf{(A.2)}] The random vector $\bX$ has a Lebesgue density with support $\mathbb{R}^d$.  
\item[\textbf{(A.3)}] Fix $\bL_d \in \{1,2,\ldots,9\}^d$.  Suppose that $Y|\bX \sim \mathrm{EF}_{g,B}\bigl(g^{-1}(f_0(\bX)),\phi_0\bigr)$, where $f_0 \in \mathcal{F}^{\bL_d}$ and $\phi_0 \in (0,\infty)$ denote the true prediction function and dispersion parameter respectively. 
\item[\textbf{(A.4)}] $f_0$ is continuous on $\mathbb{R}^d$.
\end{enumerate}

We are now in the position to state our main consistency result:
\begin{thm}
\label{Thm:SCMLEconsistency}
Assume \textbf{(A.1)} -- \textbf{(A.4)}. Then, for every $a_0 \ge 0$,
\[
	\sup_{\hat{f}_n \in \hat{S}_n^{\bL_d}} \sup_{\bx \in [-a_0, a_0]^d} |\hat{f}_n(\bx) - f_0(\bx)| \stackrel{a.s.}{\rightarrow} 0
\]
as $n \rightarrow \infty$.
\end{thm}
\noindent\textbf{Remarks}: 
\begin{enumerate}[1.]
\setlength{\itemsep}{0pt}
\setlength{\parskip}{0pt}
\setlength{\parsep}{0pt}
\item When the EF distribution is Gaussian, SCMLE coincides with the shape-constrained least squares estimator (SCLSE).  Using essentially the same argument, one can prove the the same consistency result for the SCLSE under a slightly different setting where $Y_i = f_0(\bX_i) + \epsilon_i$ for $i=1,\ldots,n$, and where $\epsilon_1,\ldots,\epsilon_n$ are independent and identically distributed with zero mean and finite variance, but are not necessarily Gaussian.
\item Assumption \textbf{(A.2)} can be weakened at the expense of lengthening the proof still further.  For instance, one can assume only that the support $\mathrm{supp}(\bX)$ of the covariates to be a convex subset of $\mathbb{R}^d$ with positive Lebesgue measure. In that case, it can be concluded that $\hat{f}_n$ converges uniformly to $f_0$ almost surely on any compact subset contained in the interior of $\mathrm{supp}(\bX)$.  In fact, with some minor modifications, our proof can also be generalised to situations where some components of $\bX$ are discrete.  
\item Even without Assumption \textbf{(A.4)}, consistency under a weaker norm can be established, namely
\[
	\sup_{\hat{f}_n \in \hat{S}_n^{\bL_d}}  \; \int_{[-a_0, a_0]^d} |\hat{f}_n(\bx) - f_0(\bx)|\, d\bx \stackrel{a.s.}{\rightarrow} 0, \quad\mbox{ as } n \rightarrow \infty.
\]
\item Instead of assuming a single dispersion parameter $\phi_0$ as done here, one can take $\phi_{ni} = \phi_0 / w_{ni}$ for $i =1,\ldots,n$, where $w_{ni}$ are known, positive weights (this is frequently needed in practice in the Binomial setting).  In that case, the new partial log-likelihood can be viewed as a weighted version of the original one. Consistency of SCMLE can be established provided that $\liminf_{n \rightarrow \infty} \frac{\min_i w_{ni}}{\max_i w_{ni}} > 0$. 
\end{enumerate}
Under assumption \textbf{(A.3)}, we may write $f_0 \stackrel{\mathcal{F}^{\bL_d}}{\sim} (f_{0,1},\ldots,f_{0,d},c_0)$.  From the proof of Theorem~\ref{Thm:SCMLEconsistency}, we see that for any $a_0 > 0$, with probability one, for sufficiently large $n$, any $\hat{f}_n \in \hat{S}_n^{\bL_d}$ can be written in the form $\hat{f}_n(\bx) = \sum_{j=1}^d \hat{f}_{n,j}(x_j) + \hat{c}_n$ for $\bx \in [-a_0,a_0]^d$, where $\hat{f}_{n,j}$ satisfies the shape constraint $l_j$ and $\hat{f}_{n,j}(0) = 0$ for each $j=1,\ldots,d$.  
The following corollary establishes the important fact that each additive component (as well as the intercept term) is estimated consistently by SCMLE.
\begin{cor}
\label{Cor:SCMLEconsistency}
Assume \textbf{(A.1)} -- \textbf{(A.4)}. Then, for any $a_0 \ge 0$,
\[
	  \sup_{\hat{f}_n \in \hat{S}_n^{\bL_d}} \bigg\{ \sum_{j=1}^d \sup_{x_j \in [-a_0, a_0]}|\hat{f}_{n,j}(x_j) - f_{0,j}(x_j)| + |\hat{c}_n-c_0| \bigg\} \stackrel{a.s.}{\rightarrow} 0
\]
as $n \rightarrow \infty$.
\end{cor}

\section{Generalised additive index models with shape constraints}
\label{Sec:gaim}
\subsection{The generalised additive index model and its identifiability}
Recall that in the generalised additive index model, the response $Y_i \in \mathbb{R}$ and the predictor $\bX_i = (X_{i1},\ldots,X_{id})^T \in \mathbb{R}^d$ are related through (\ref{Eq:gaim}), where $g$ is a known link function, and where conditional on $\bX_i$, the response $Y_i$ has a known EF distribution with mean parameter $g^{-1}(f(\bX_i))$ and dispersion parameter $\phi$.

Let $\bA=(\balpha_1,\ldots,\balpha_m)$ denote the $d \times m$ \emph{index matrix}, where $m \leq d$, and let $f(\bz) = \sum_{j=1}^m f_j(z_j)+c$ for $\bz = (z_1,\ldots,z_m)^T \in \mathbb{R}^m$, so the prediction function can be written as $f^I(\bx) = f(\bA^T \bx)$ for $\bx = (x_1,\ldots,x_d)^T \in \mathbb{R}^d$.   As in Section~\ref{Sec:gam}, we impose shape constraints on the ridge functions by assuming that $f_j: \mathbb{R} \rightarrow \mathbb{R}$ satisfies the shape constraint with label $l_j \in \{1,2,\ldots,9\}$, for $j=1,\ldots,m$.

To ensure the identifiability of the model, we only consider additive index functions $f^I$ of the form (\ref{Eq:gaim}) satisfying the following conditions, adapted from \citet{Yuan2011}:
\begin{enumerate}
\setlength{\itemsep}{0pt}
\setlength{\parskip}{0pt}
\setlength{\parsep}{0pt}
\item[{\textbf{(B.1a)}}] $f_j(0) = 0$ for $j = 1, \ldots,m$.
\item[{\textbf{(B.1b)}}] $\|\balpha_j\|_1 = 1$ for $j = 1, \ldots, m$, where $\|\cdot\|_1$ denotes the $\ell_1$ norm.
\item[{\textbf{(B.1c)}}] The first non-zero entry of $\balpha_j$ is positive for every $j$ with $l_j \in \{1,4,7\}$.
\item[{\textbf{(B.1d)}}] There is at most one linear ridge function in $f_1,\ldots,f_m$; if $f_k$ is linear, then $\balpha_{j}^T \balpha_{k} = 0$ for every $j \neq k$.
\item[{\textbf{(B.1e)}}] There is at most one quadratic ridge function in $f_1,\ldots,f_m$.
\item[{\textbf{(B.1f)}}] $\bA$ has full column rank $m$. 
\end{enumerate}


\subsection{GAIM  estimation}
\label{Sec:gaimMLE}
Let $\bA_0 = (\balpha_{0,1},\ldots,\balpha_{0,m})$ denote the true index matrix.  For $\bx = (x_1,\ldots,x_d)^T \in \mathbb{R}^d$, let 
\[
f_0^I(\bx) = f_{0,1}(\balpha_{0,1}^T \bx) + \ldots +  f_{0,m}(\balpha_{0,m}^T \bx) + c_0 
\] 
be the true prediction function, and write $f_0(\bz) = \sum_{j=1}^m f_{0,j}(z_j)+c_0$ for $\bz = (z_1,\ldots,z_m)^T \in \mathbb{R}^m$.  Again we restrict our attention to the common EF distributions listed in Table~\ref{Tab:glmformula} and take $g$ to be the corresponding canonical link function.  Let 
\begin{align*}
	\mathcal{A}^{\bL_m}_d = \Big\{&\bA = (\balpha_1,\ldots,\balpha_m) \in \mathbb{R}^{d \times m}: \bA \mbox{ satisfies assumptions } \textbf{(B.1b)}-\textbf{(B.1c)},\\
	&\mbox{and if there exists } k \in \{1,\ldots,m\} \mbox{ s.t. } l_k = 1, \mbox{ then }  \balpha_j^T \balpha_k = 0 \mbox{ for every } j \neq k \Big\}.
\end{align*}
Given a shape vector $\bL_m$, we consider the set of shape-constrained additive index functions given by 
\begin{align*}
	\mathcal{G}^{\bL_m}_d = \Big\{ f^I:\mathbb{R}^d\rightarrow \mathbb{R} \; \Big| \; & f^I(\bx) = f(\bA^T \bx), \mbox{ with } f \in \mathcal{F}^{\bL_m} \mbox{ and } \bA \in \mathcal{A}^{\bL_m}_d \Big\},
\end{align*}
A natural idea is to seek to maximise the scaled partial log-likelihood $\bar{\ell}_{n,d}$ over the pointwise closure of $\mathcal{G}^{\bL_m}_d$. As part of this process, one would like to find a $d \times m$ matrix in $\mathcal{A}^{\bL_m}_d$ that maximises the scaled partial \emph{index log-likelihood} 
\[
	\Lambda_n(\bA) \equiv \Lambda_n \Big(\bA; (\bX_1,Y_1),\ldots,(\bX_n,Y_n) \Big) = \sup_{f \in \mathcal{F}^{\bL_m}} \bar{\ell}_{n,m}\Big(f; (\bA^T \bX_1,Y_1),\ldots,(\bA^T \bX_n,Y_n) \Big),
\]
where the dependence of $\Lambda_n(\cdot)$ on $\bL_m$ is suppressed for notational convenience.  We argue, however, that this strategy has two drawbacks: 
\begin{enumerate}[1.]
\setlength{\itemsep}{0pt}
\setlength{\parskip}{0pt}
\setlength{\parsep}{0pt}
\item \textbf{`Saturated' solution}. In certain cases, maximising $\Lambda_n(\bA)$ over $\mathcal{A}^{\bL_m}_d$ can lead to a perfect fit of the model. We demonstrate this phenomenon via the following example. 
\begin{example}
Consider the Gaussian family with the identity link function. We take $d=2$. Assume that there are $n$ observations $(\bX_1,Y_1),\ldots,(\bX_n,Y_n)$ with $\bX_i = (X_{i1},X_{i2})^T$ and that $\bL_2=(2,3)^T$. We assume here that $X_{11} < \ldots < X_{n1}$. Note that it is possible to find an increasing function $f_1$, an decreasing function $f_2$ (with $f_1(0) = f_2(0) = 0$) and a constant $c$ such that $f_1(X_{i1}) + f_2(X_{i1}) + c = Y_i$ for every $i = 1,\ldots, n$. Now pick $\epsilon$ such that 
\[
0 < \epsilon < \min\bigg\{\frac{1}{2},\; \frac{\min_{1 \le i < n} (X_{i+1,1} - X_{i1})}{4 (\max_{1 \le i \le n} |X_{i2}| + 1) }\bigg\},
\]
and let $\bA = (\balpha_1,\balpha_2) = \begin{pmatrix} 1 & 1-\epsilon \\ 0 & \epsilon \end{pmatrix}$. It can be checked that $\{\balpha_2^T \bX_i\}_{i = 1}^n$ is a strictly increasing sequence, so one can find a decreasing function $f_2^*$ such that $f_2^*(\balpha_2^T \bX_i) = f_2 (X_{i1})$ for every $i = 1,\ldots, n$. Consequently, by taking $\hat{f}^I(\bx) = f_1(\bA^T \bx) + f_2^*(\bA^T \bx) + c$, we can ensure that $\hat{f}^I(\bX_i) = Y_i$ for every $i = 1,\ldots,n$.
\end{example}
We remark that this `perfect-fit' phenomenon is quite general. Actually, one can show (via simple modifications of the above example) that it could happen whenever $\bL_m \notin \mathcal{L}_m$, where $\mathcal{L}_m = \{1,\ldots,9\}$ when $m=1$, and $\mathcal{L}_m = \{1,4,5,6\}^m \cup \{1,7,8,9\}^m$ when $m \ge 2$.

\item \textbf{Lack of upper semi-continuity of $\Lambda_n$}. The function $\Lambda_n(\cdot)$ need not be upper-semicontinuous, as illustrated by the following example: 
\begin{example}
Again consider the Gaussian family with the identity link function. Take $d=2$ and $\bL_2=(2,2)^T$. Assume that there are $n=4$ observations, namely, $\bX_1 = (0,0)^T$, $\bX_2 = (0,1)^T$, $\bX_3 = (1,0)^T$, $\bX_4=(1,1)^T$, $Y_1 = Y_2 = Y_3 = 0$ and $Y_4 = 1$. If we take $\bA = \begin{pmatrix} 1 & 0 \\ 0 & 1 \end{pmatrix}$, then it can be shown that $\Lambda_n(\bA) = 3/32$ by fitting $\hat{f}^I(\bX_1) = -\frac{1}{4}$, $\hat{f}^I(\bX_2) = \hat{f}^I(\bX_3) = 1/4$ and $\hat{f}^I(\bX_4) = 3/4$. However, for any sufficiently small $\epsilon >0$, if we define $\bA_\epsilon = \begin{pmatrix} 1-\epsilon & -\epsilon \\ -\epsilon & 1-\epsilon \end{pmatrix}$, then we can take $\hat{f}^I(\bX_i) = Y_i$ for $i = 1,\ldots,4$, so that $\Lambda_n(\bA_\epsilon) = 1/8 > \Lambda_n(\bA)$.
\end{example}
This lack of upper semi-continuity means in general we cannot guarantee the existence of a maximiser. 
\end{enumerate}

As a result, certain modifications are required for our shape-constrained approach to be successful in the context of GAIMs.
To deal with the first issue when $\bL_m \notin \mathcal{L}_m$, we optimise $\Lambda_n(\cdot)$ over the subset of matrices
\begin{align*}
	\mathcal{A}^{\bL_m,\delta}_d = \big\{ \bA \in \mathcal{A}^{\bL_m}_d:  \lambda_{\min}(\bA^T \bA) \ge \delta \big\}
\end{align*}
for some pre-determined $\delta > 0$, where $\lambda_{\min}(\cdot)$ denotes the smallest eigenvalue of a non-negative definite matrix. Other strategies are also possible. For example, when $\bL_m = (2,\ldots,2)^T$, the `perfect-fit' phenomenon can be avoided by only considering matrices that have the same signs in all entries (cf. Section~\ref{Sec:Javelin} below).

To address the second issue, we will show that given $f_0^I \in \mathcal{G}^{\bL_m}_d$ satisfying the identifiability conditions, to obtain a consistent estimator, it is sufficient to find $\tilde{f}^I_n$ from the set
\begin{align}
\label{Eq:SCAIE}
	\tilde{S}^{\bL_m}_n \in \bigg\{ f^I: &\mathbb{R}^d \rightarrow \mathbb{R} \Big| \; f^I(\bx) = f(\bA^T \bx), \mbox{ with } f \in \mathcal{F}^{\bL_m}; \\
\notag	& \mbox{if } \bL_m \in \mathcal{L}_m, \mbox{ then } \bA \in \mathcal{A}^{\bL_m}_d, \mbox{ otherwise}, \bA \in \mathcal{A}^{\bL_m,\delta}_d; \\ 
\notag	& \bar{\ell}_{n,m}\Big(f;(\bA^T \bX_1,Y_1)\ldots,(\bA^T \bX_n,Y_n)\Big) \geq \bar{\ell}_{n,m}\Big(f_0;(\bA_0^T \bX_1,Y_1)\ldots,(\bA_0^T \bX_n,Y_n)\Big)\bigg\},
\end{align}
for some $\delta \in (0,\lambda_{\mathrm{min}}(\bA_0^T \bA_0)]$. We write $\tilde{f}_n^I(\bx) = \tilde{f}_n(\tilde{\bA}_n^T \bx)$, where $\tilde{\bA}_n = (\tilde{\balpha}_{n,1},\ldots,\tilde{\balpha}_{n,m}) \in \mathcal{A}^{\bL_m}_d \mbox{ or } \mathcal{A}^{\bL_m,\delta}_d$ is the estimated index matrix and $\tilde{f}_n(\bz) = \sum_{j=1}^m \tilde{f}_{n,j}(z_j) + \tilde{c}_n$ is the estimated additive function satisfying $\tilde{f}_{n,j}(0) = 0$ for every $j = 1,\ldots,m$. We call $\tilde{f}^I_n$ the \emph{shape-constrained additive index estimator} (SCAIE), and write $\tilde{\bA}_n$ and $\tilde{f}_{n,1},\ldots,\tilde{f}_{n,m}$ respectively for the corresponding estimators of the index matrix and ridge functions. 

When there exists a maximiser of the function $\Lambda_n(\cdot)$ over $\mathcal{A}^{\bL_m}_d$ or $\mathcal{A}^{\bL_m,\delta}_d$, the set $\tilde{S}^{\bL_m}_n$ is non-empty; otherwise, a function satisfying (\ref{Eq:SCAIE}) still exists in view of the following proposition: 
\begin{prop}
\label{Prop:SCAIEcontinuity}
The function $\Lambda_n(\cdot)$ is lower-semicontinuous.
\end{prop}
Note that if a maximiser of $\Lambda_n(\cdot)$ does not exist, there must exist some $\mathring{\bA}_n$ such that $\Lambda_n(\mathring{\bA}_n) > \Lambda_n(\bA_0)$. It then follows from Proposition~\ref{Prop:SCAIEcontinuity} that 
\[
\liminf_{\bA \rightarrow \mathring{\bA}_n} \Lambda_n(\bA) \ge \Lambda_n(\mathring{\bA}_n) > \Lambda_n(\bA_0) \ge \bar{\ell}_{n,m}\big(f_0;(\bA_0^T \bX_1,Y_1)\ldots,(\bA_0^T \bX_n,Y_n)\big),
\]
so any $\bA$ sufficiently close to $\mathring{\bA}_n$ yields a prediction function in $\tilde{S}^{\bL_m}_n$. A stochastic search algorithm can be employed to find such matrices; see Section~\ref{Sec:computeSCAIE} for details.

\subsection{Consistency of SCAIE}
\label{Sec:indextheory}
In this subsection, we show the consistency of $\hat{f}^I_n$ under a random design setting. In addition to \textbf{(A.1)} -- \textbf{(A.2)}, we require the following conditions:
\begin{enumerate}
\setlength{\itemsep}{0pt}
\setlength{\parskip}{0pt}
\setlength{\parsep}{0pt}
\item[{\textbf{(B.2)}}] The true prediction function $f_0^I$ belongs to $\mathcal{G}^{\bL_m}_d$.
\item[{\textbf{(B.3)}}] Fix $\bL_m \in \{1,2,\ldots,9\}^m$.  Suppose that $Y|\bX \sim \mathrm{EF}_{g,B}\bigl(g^{-1}(f_0^I(\bX)),\phi_0\bigr)$, where $\phi_0 \in (0,\infty)$ is the true dispersion parameter.
\item[{\textbf{(B.4)}}] $f_0^I$ is continuous on $\mathbb{R}^d$.
\item[{\textbf{(B.5)}}] $f_0^I$ and the corresponding index matrix $\bA_0$ satisfy the identifiability conditions \textbf{(B.1a)} -- \textbf{(B.1f)}.
\end{enumerate}
\begin{thm}
\label{Thm:SCAIEconsistency}
Assume \textbf{(A.1)} -- \textbf{(A.2)} as well as \textbf{(B.2)} -- \textbf{(B.5)}. Then, provided $\delta \leq \lambda_{\mathrm{min}}(\bA_0^T \bA_0)$ when $\bL_m \notin \mathcal{L}_m$, we have for every $a_0 \ge 0$ that
\[
	\sup_{\tilde{f}_n^I \in \tilde{S}^{\bL_m}_n } \sup_{\bx \in [-a_0, a_0]^d} |\tilde{f}_n^I(\bx) - f_0^I(\bx)| \stackrel{a.s.}{\rightarrow} 0, \quad\mbox{ as } n \rightarrow \infty.
\]
\end{thm}
Consistency of the estimated index matrix and the ridge functions is established in the next corollary. 
\begin{cor}
\label{Cor:SCAIEconsistency}
Assume \textbf{(A.1)} -- \textbf{(A.2)} and \textbf{(B.2)} -- \textbf{(B.5)}.  Then, provided $\delta \leq \lambda_{\mathrm{min}}(\bA_0^T \bA_0)$ when $\bL_m \notin \mathcal{L}_m$, we have for every $a_0 \ge 0$ that  
\[
	\sup_{\tilde{f}_n^I \in \tilde{S}^{\bL_m}_n } \min_{\tilde{\pi}_n \in \mathcal{P}_m} \biggl\{\sum_{j=1}^m \|\tilde{\balpha}_{n,\tilde{\pi}_n(j)}-\balpha_{0,j}\|_1 + \sum_{j=1}^m \sup_{z_j \in [-a_0, a_0]}|\tilde{f}_{n,\tilde{\pi}_n(j)}(z_j) - f_{0,j}(z_j)| + |\tilde{c}_n-c_0| \biggr\} \stackrel{a.s.}{\rightarrow} 0, 
\]
as $n \rightarrow \infty$, where $\mathcal{P}_m$ denotes the set of permutations of $\{1,\ldots,m\}$.
\end{cor}
Note that we can only hope to estimate the set of projection indices, and not their ordering (which is arbitrary).  This explains why we take the minimum over all permutations of $\{1,\ldots,m\}$ in Corollary~\ref{Cor:SCAIEconsistency} above.

\section{Computational aspects}
\label{Sec:compute}
\subsection{Computation of SCMLE}
\label{Sec:computeSCMLE}
Throughout this subsection, we fix $\bL_d = (l_1,\ldots,l_d)^T$, the EF distribution and the values of the observations, and present an algorithm for computing SCMLE described in Section~\ref{Sec:gam}.  We seek to reformulate the problem as convex program in terms of basis functions and apply an active set algorithm \citep{NocedalWright2006}.  Such algorithms have recently become popular for computing various shape-constrained estimators. For instance, \citet{GJW2008} used a version, which they called the `support reduction algorithm' in the one-dimensional convex regression setting; \citet{DumbgenRufibach2011} applied another variant to compute the univariate log-concave maximum likelihood density estimator. Recently, \citet{Meyer2013b} developed a `hinge' algorithm for quadratic programming, which can also be viewed as a variant of the active set algorithm. 

Without loss of generality, we assume in the following that only the first $d_1$ components ($d_1 \le d$) of $f_0$ are linear, i.e. $l_1 = \cdots = l_{d_1} = 1$ and $(l_{d_1+1},\ldots,l_d)^T \in \{2,\ldots,9\}^{d-d_1}$. Furthermore, we assume that the order statistics $\{X_{(i),j}\}_{i=1}^n$ of $\{X_{ij}\}_{i=1}^n$ are distinct for every $j = d-d_1+1,\ldots,d$.  Fix $\bx = (x_1,\ldots,x_d)^T \in \mathbb{R}^d$ and define the basis functions $g_{0j}(x_j) = x_j$ for $j=1,\ldots,d_1$ and, for $i=1,\ldots,n$, 
\begin{align*}
g_{ij}(x_j) = 
\begin{cases}
  \mathbbm{1}_{\{X_{(i),j} \le x_j\}} - \mathbbm{1}_{\{X_{(i),j} \le 0\}}, & \quad \mbox{if } l_j = 2, \\
  \mathbbm{1}_{\{x_j  <  X_{(i),j}\}} - \mathbbm{1}_{\{0  <  X_{(i),j}\}}, & \quad \mbox{if } l_j = 3, \\
  (x_j - X_{(i),j}) \mathbbm{1}_{\{X_{(i),j} \le x_j\}} + X_{(i),j} \mathbbm{1}_{\{X_{(i),j} \le 0\}}, & \quad \mbox{if } l_j = 4 \mbox{ or } l_j = 5, \\
  (X_{(i),j} - x_j) \mathbbm{1}_{\{x_j \le X_{(i),j}\}} - X_{(i),j} \mathbbm{1}_{\{0 \le X_{(i),j}\}}, & \quad \mbox{if } l_j = 6,  \\
  (X_{(i),j} - x_j) \mathbbm{1}_{\{X_{(i),j} \le x_j\}} - X_{(i),j} \mathbbm{1}_{\{X_{(i),j} \le 0\}}, & \quad \mbox{if } l_j = 7 \mbox{ or } l_j = 9, \\
  (x_j - X_{(i),j}) \mathbbm{1}_{\{x_j \le X_{(i),j}\}} + X_{(i),j} \mathbbm{1}_{\{0 \le X_{(i),j}\}}, & \quad \mbox{if } l_j = 8.
\end{cases} 
\end{align*}
Note that all the basis functions given above are zero at the origin.  Let $\mathcal{W}$ denote the set of weight vectors
\[
	\bw = (w_{00},w_{01},\ldots,w_{0d_1}, w_{1(d_1+1)},\ldots,w_{n(d_1+1)},\ldots,w_{1d},\ldots,w_{nd})^T \in \mathbb{R}^{n(d-d_1)+d_1+1}
\]
satisfying
\begin{align*}
\begin{cases}
  w_{ij} \ge 0, & \text{for every } i = 1, \ldots, n \text{ and every } j \text{ with } l_j \in \{2,3,5,6,8,9\} \\
  w_{ij} \ge 0, & \text{for every } i = 2, \ldots, n \text{ and every } j \text{ with } l_j \in \{4,7\}.
\end{cases}
\end{align*}
To compute SCMLE, it suffices to consider prediction functions of the form
\[
	f^{\bw}(\bx) = w_{00} + \sum_{j=1}^{d_1} w_{0j}g_{0j}(x_j) + \sum_{j=d_1+1}^d \sum_{i=1}^n w_{ij} g_{ij}(x_j)
\]
subject to $\bw \in \mathcal{W}$.  Our optimisation problem can then be reformulated as maximising 
\[
	\psi_n(\bw) = \bar{\ell}_{n,d}(f^{\bw}; (\bX_1,Y_1), \ldots,(\bX_n,Y_n))
\]
over $\mathbf{w} \in \mathcal{W}$. Note that $\psi_n$ is a concave (but not necessarily strictly concave) function. Since
\[
	\sup_{\bw \in \mathcal{W}} \bar{\ell}_{n,d}(f^{\bw}) = \bar{\ell}_{n,d}(\hat{f}_n),
\]
our goal here is to find a sequence $(\bw^{(k)})$ such that $\psi_n(\bw^{(k)}) \rightarrow \sup_{\bw \in \mathcal{W}} \bar{\ell}_{n,d}(f^{\bw})$ as $k \rightarrow \infty$.  In Table~\ref{Tab:codeSCMLE}, we give the pseudo-code for our active set algorithm for finding SCMLE, which is implemented in the \texttt{R} package \texttt{scar} \citep{ChenSamworth2014}.
\begin{table}[!htbp]
\begin{tabular}{ l p{14cm} }
\hline 
  \textbf{Step 1}: &  \textbf{Initialisation - outer loop}: sort $\{X_i\}_{i=1}^n$ coordinate by coordinate; define the initial working set as $\mathcal{S}_1 = \{(0,j) | j \in \{1,\ldots,d_1\}\} \cup \{(1,j) | l_j \in \{4,7\}\}$; in addition, define the set of potential elements as 
\[
 \mathcal{S} = \Big\{(i,j): \; i = 1,\ldots,n, j = d_1+1 ,\ldots, d\Big\};
\] 
set the iteration count $k = 1$.\\
  \textbf{Step 2}: &  \textbf{Initialisation - inner loop}: if $k > 1$, set $\bw^* = \bw^{(k-1)}$.\\
  \textbf{Step 3}: &  \textbf{Unrestricted generalised linear model (GLM)}: solve the following unrestricted GLM problem using iteratively reweighted least squares (IRLS):
\[
  \frac{1}{n}\sum_{h=1}^n \biggl\{ Y_h \Big(\sum_{(i,j) \in \mathcal{S}_k} w_{ij} g_{ij}(X_{hj}) + w_{00}\Big) - B\Big(\sum_{(i,j) \in \mathcal{S}_k} w_{ij} g_{ij}(X_{hj}) + w_{00} \Big) \biggr\},
\]
where for $k > 1$, $\bw^*$ is used as a warm start. Store its solution in $\bw^{(k)}$ (with zero weights for the elements outside $\mathcal{S}_k$).\\
  \textbf{Step 4}: &  \textbf{Working set refinement}: if $k=1$ or if $w_{ij} > 0$ for every $(i,j) \in \mathcal{S}_k \backslash \mathcal{S}_1$, go to \textbf{Step~5}; otherwise, define respectively the moving ratio $p$ and the set of elements to drop as 
\[
  p = \min_{\abovebelow{(i,j) \in \mathcal{S}_k \backslash \mathcal{S}_1:}{w_{ij} \le 0}} \frac{w_{ij}^*}{w_{ij}^* - w_{ij}},\; \mathcal{S}_- = \biggl\{(i,j): \; (i,j) \in \mathcal{S}_k \backslash \mathcal{S}_1, \, w_{ij} \le 0, \, \frac{w_{ij}^*}{w_{ij}^* - w_{ij}} =  p\biggr\},
\] 
set $\mathcal{S}_k := \mathcal{S}_k \backslash \mathcal{S}_-$, overwrite $\bw^*$ by $\bw^* := (1-p)\bw^* + p\bw^{(k)}$ and go to \textbf{Step~3}. \\
  \textbf{Step 5}: &  \textbf{Derivative evaluation}: for every $(i,j) \in \mathcal{S}$, compute $D_{i,j}^{(k)} = \frac{\partial \psi_n}{\partial w_{ij}}(\bw^{(k)})$.\\
  \textbf{Step 6}: &  \textbf{Working set enlargement}: write $\mathcal{S_+} = \argmax_{(i,j)\in \mathcal{S}} D_{i,j}^{(k)}$ for the enlargement set, with maximum $D^{(k)} = \max_{(i,j) \in \mathcal{S}} D_{i,j}^{(k)}$; if $D^{(k)} \le 0$ (or some other criteria are met if the EF distribution is non-Gaussian, e.g. $D^{(k)} < \epsilon_{IRLS}$ for some pre-determined small $\epsilon_{IRLS} > 0$), STOP the algorithm and go to \textbf{Step~7}; otherwise, pick any single-element subset $\mathcal{S}^*_{+} \subseteq \mathcal{S}_+$, let $\mathcal{S}_{k+1} = \mathcal{S}_k \cup \mathcal{S}^*_{+}$, set $k := k+1$ and go back to \textbf{Step~2}.\\
  \textbf{Step 7}: &  \textbf{Output}: for every $j = 1,\ldots,d$, set $\hat{f}_{n,j}(x_j) = \sum_{\{i:(i,j) \in \mathcal{S}_k\}} w^{(k)}_{ij} g_{ij}(x_j)$; take $\hat{c}_n = w_{00}^{(k)}$; finally, return SCMLE as $\hat{f}_n(\bx) = \sum_{j=1}^d \hat{f}_{n,j}(x_j) + \hat{c}_n$. \vspace{2mm}\\
\hline
\end{tabular}
\caption{Pseudo-code of the active set algorithm for computing SCMLE}
\label{Tab:codeSCMLE}
\end{table}
We outline below some implementation details: 
\begin{enumerate}[(a)]
\item \textbf{IRLS}. Step~3 solves an unrestricted GLM problem by applying iteratively reweighted least squares (IRLS).  Since the canonical link function is used here, IRLS is simply the Newton--Raphson method. If the EF distribution is Gaussian, then IRLS gives the exact solution of the problem in just one iteration. Otherwise, there is no closed-form expression for the solution, so a threshold $\epsilon_{IRLS}$ has to be picked to serve as part of the stopping criterion. Note that here IRLS can be replaced by other methods that solve GLM problems, though we found that IRLS offers competitive timing performance.

\item \textbf{Fast computation of the derivatives}. Although Step~5 appears at first sight to require $O(n^2 d)$ operations, it can actually be completed with only $O(nd)$ operations by exploiting some nice recurrence relations. Define the `nominal' residuals at the $k$-th iteration by 
\[
r_i^{(k)} = Y_i - \mu_i^{(k)}, \quad\mbox{ for } i = 1,\ldots,n,
\]
where $\mu_i^{(k)} = g^{-1}(f^{\bw^{(k)}}(\bX_i))$ are the fitted mean values at the $k$-th iteration.  Then
\[
\frac{\partial \psi_n}{\partial w_{ij}}(\bw^{(k)}) = \frac{1}{n} \sum_{u=1}^n r_u^{(k)} g_{ij}(X_{uj}).
\]
For simplicity, we suppress henceforth the superscript $k$. Now fix $j$ and reorder the pairs $(r_i, X_{ij})$ as $(r_{(1)}, X_{(1),j}),\ldots,(r_{(n)}, X_{(n),j})$ such that $X_{(1),j} \le \ldots \le X_{(n),j}$ (note that this is performed in Step~1).  Furthermore, define 
\begin{align*}
R_{i,j} = 
\begin{cases}
 \sum_{u=1}^{i} r_{(u)}, &\mbox{ if } l_j \in \{2,4,5,6\}, \\
 -\sum_{u=1}^{i} r_{(u)}, &\mbox{ if } l_j \in \{3,7,8,9\},\\
\end{cases}
\end{align*}
for $i = 1,\ldots,n$, where we suppress the explicit dependence of $r_{(u)}$ on $j$ in the notation.  We have $R_{n,j} = 0$ due to the presence of the intercept $w_{00}$.  The following recurrence relations can be derived by simple calculation:
\begin{itemize}
\item For $l_j \in \{2,3\}$, we have $D_{1,j} = 0$ and $n D_{i,j} = -R_{i-1,j}$ for $i = 2,\ldots,n$.
\item For $l_j \in \{4,5,7,9\}$, the initial condition is $D_{n,j} = 0$, and
\begin{align*}
n D_{i,j} =  n D_{i+1,j} - R_{i,j} \ (X_{(i+1),j} - X_{(i),j}), &\mbox{ for } i = n-1,\ldots,1.
\end{align*}
\item For $l_j \in \{6,8\}$, the initial condition is $D_{1,j} = 0$, and
\begin{align*}
n D_{i,j} = n D_{i-1,j} + R_{i-1,j} \ (X_{(i),j} - X_{(i-1),j}), &\mbox{ for } i = 2,\ldots,n.
\end{align*}
\end{itemize}
Therefore, the complexity of Step~5 in our implementation is $O(nd)$.

\item \textbf{Convergence}. If the EF distribution is Gaussian, then it follows from Theorem~1 of \citet{GJW2008} that our algorithm converges to the optimal solution after finitely many iterations.  In general, the convergence of this active set strategy depends on two aspects:
\begin{itemize}
\item \underline{Convergence of IRLS}. The convergence of Newton--Raphson method in Step~3 depends on the starting values. It is not guaranteed without step-size optimisation; cf. \citet{Jorgensen1983}.  However, starting from the second iteration, each subsequent IRLS is performed by starting from the previous well-approximated solution, which typically makes the method work well.
\item \underline{Accuracy of IRLS}. If IRLS gives the \emph{exact} solution every time, then $\psi_n(\bw^{(k)})$ increases at each iteration. In particular, one can show that at the $k$-th iteration, the new element $\mathcal{S}^*_{+}$ added into the working set in Step~6 will remain in the working set $\mathcal{S}_{k+1}$ after the $(k+1)$-th iteration. However, since IRLS only returns an approximate solution, there is no guarantee that the above-mentioned phenomenon continues to hold.  One way to resolve this issue is to reduce the tolerance $\epsilon_{IRLS}$ if $\psi_n(\bw^{(k)}) \le \psi_n(\bw^{(k-1)})$, and redo the computations for both the previous and the current iteration. 
\end{itemize}
Here we terminate our algorithm in Step~6 if either $\psi_n(\bw^{(k)})$ is non-increasing or $D^{(k)} < \epsilon_{IRLS}$.  In our numerical work, we did not encounter convergence problems, even outside the Gaussian setting.
\end{enumerate}

\subsection{Computation of SCAIE}
\label{Sec:computeSCAIE}
The computation of SCAIE can be divided into two parts:
\begin{enumerate}[1.]
\item For a given fixed $A$, find $f \in \mathrm{cl}(\mathcal{F}^{\bL_m})$ that maximises $\bar{\ell}_{n,m}\big(f; (\bA^T \bX_1,Y_1),\ldots,(\bA^T \bX_n,Y_n) \big)$ using the algorithm in Table~\ref{Tab:codeSCMLE} but with $\bA^T\bX_i$ replacing $\bX_i$.  Denote the corresponding maximum value by $\Lambda_n(\bA)$.
\item For a given lower-semicontinuous function $\Lambda_n$ on $\mathcal{A}_d^{\bL_m}$ or $\mathcal{A}_d^{\bL_m,\delta}$ as appropriate, find a maximising sequence $(\bA^k)$ in this set.
\end{enumerate}
The second part of this algorithm solves a finite-dimensional optimisation problem. Possible strategies include the differential evolution method \citep{PSL2005, DSS2011} or a stochastic search strategy \citep{DSS2013} described below. In Table~\ref{Tab:codeSCAIE}, we give the pseudo-code for computing SCAIE.  We note that Step~4 of the stochastic search algorithm is parallelisable. 
\begin{table}[!htbp]
\begin{tabular}{ l p{14cm} }
\hline 
  \textbf{Step 1}: &  \textbf{Initialisation}: let $N$ denote the total number of stochastic searches; set $k = 1$. \\
  \textbf{Step 2}: &  \textbf{Draw random matrices}: draw a $d \times m$ random matrix $\bA^k$ by initially choosing the entries to be independent and identically distributed $N(0,1)$ random variables.  For each column of $\bA^k$, if there exists a $j \in \{1,\ldots,m\}$ such that $l_j=1$, subtract its projection to the $j$-th column of $\bA^k$ so that \textbf{(B.1d)} is satisfied, then normalise each column so \textbf{(B.1b)} and \textbf{(B.1c)} are satisfied. \\
  \textbf{Step 3}: &  \textbf{Rejection sampling}: if $\bL_m \notin \mathcal{L}_m$ and $\lambda_{\mathrm{min}}((\bA^k)^T \bA^k) < \delta$, then go back to \textbf{Step~2}; otherwise, if $k < N$, set $k := k + 1$ and go to \textbf{Step 2}.\\
  \textbf{Step 4}: &  \textbf{Evaluation of $\Lambda_n$}: for every $k = 1,\ldots, N$, compute $\Lambda_n(\bA^k)$ using the active set algorithm described in Table~\ref{Tab:codeSCMLE}. \\
  \textbf{Step 5}: &  \textbf{Index matrix estimation - 1}: let $\bA^* \in \argmax_{1 \le k \le N} \Lambda_n(\bA^k)$; set $\tilde{\bA}_n = \bA^*$; \\
  \textbf{Step 6}: &  \textbf{Index matrix estimation - 2 (optional)}: treat $\bA^*$ as a warm-start and apply another optimisation strategy to find $\bA^{**}$ in a neighbourhood of $\bA^*$ such that $\Lambda_n(\bA^{**}) > \Lambda_n(\bA^{*})$; if such $\bA^{**}$ can be found,  set $\tilde{\bA}_n = \bA^{**}$.\\ 
  \textbf{Step 7}: &  \textbf{Output}: use the active set algorithm described in Table~\ref{Tab:codeSCMLE} to find 
\[
\tilde{f}_n \in \argmax_{f \in \mathrm{cl}(\mathcal{F}^{\bL_m})} \bar{\ell}_{n,m}\big(f; (\tilde{\bA}_n^T \bX_1,Y_1),\ldots,(\tilde{\bA}_n^T \bX_n,Y_n) \big);
\] 
finally, output SCAIE as $\tilde{f}_n^I(\bx) = \tilde{f}_n(\tilde{\bA}_n^T \bx)$. \vspace{2mm}\\
\hline
\end{tabular}
\caption{Pseudo-code of the stochastic search algorithm for computing SCAIE}
\label{Tab:codeSCAIE}
\end{table}

\section{Simulation study}
\label{Sec:sim}

To analyse the empirical performance of SCMLE and SCAIE, we ran a simulation study focusing on the running time and the predictive performance. Throughout this section, we took $\epsilon_{IRLS} = 10^{-8}$.

\subsection{Generalised additive models with shape restrictions}
\label{Sec:simSCMLE}
We took $\bX_1,\ldots,\bX_n \stackrel{\mathrm{i.i.d.}}{\sim} U[-1,1]^d$. The following three problems were considered:

\begin{enumerate}[1.]
\setlength{\itemsep}{0pt}
\setlength{\parskip}{0pt}
\setlength{\parsep}{0pt}
\item Here $d=4$. We set $\bL_4=(4,4,4,4)^T$ and $f_0(\bx)  = |x_1| + |x_2| + |x_3|^3 + |x_4|^3$.

\item Here $d=4$. We set $\bL_4=(5,5,5,5)^T$ and 
\[
	f_0(\bx)  = x_1 \mathbbm{1}_{\{x_1 \ge 0\}}  +  x_2 \mathbbm{1}_{\{x_2 \ge 0\}}  +  x_3^3 \mathbbm{1}_{\{x_3 \ge 0\}}  +  x_4^3 \mathbbm{1}_{\{x_4 \ge 0\}}.
\] 

\item Here $d=8$. We set $\bL_8=(4,4,4,4,5,5,5,5)^T$ and 
\[
f_0(\bx)  = |x_1|  + |x_2|  + |x_3|^3  + |x_4|^3 + x_5 \mathbbm{1}_{\{x_5 \ge 0\}} + x_6 \mathbbm{1}_{\{x_6 \ge 0\}} + x_7^3 \mathbbm{1}_{\{x_7 \ge 0\}} + x_8^3 \mathbbm{1}_{\{x_8 \ge 0\}}.
\]
\end{enumerate}

For each of these three problems, we considered three types of EF distributions: 
\begin{itemize}
\setlength{\itemsep}{0pt}
\setlength{\parskip}{0pt}
\setlength{\parsep}{0pt}
\item \underline{Gaussian}: for $i = 1, \ldots, n$, conditional on $\bX_i$, draw independently $Y_i \sim N(f_0(\bX_i),0.5^2)$;
\item \underline{Poisson}:  for $i = 1, \ldots, n$, conditional on $\bX_i$, draw independently $Y_i \sim \mathrm{Pois} \big(g^{-1}(f_0(\bX_i)) \big)$, where $g(\mu) = \log \mu$; 
\item \underline{Binomial}: for $i = 1, \ldots, n$, draw $N_i$ (independently of $\bX_1,\ldots,\bX_n$) from a uniform distribution on $\{11, 12, \ldots, 20\}$, and then draw independently $Y_i \sim N_i^{-1}\mathrm{Bin}\big(N_i, g^{-1}(f(\bX_i))\big)$, where $g(\mu) = \log \frac{\mu}{1-\mu}$.
\end{itemize}
Note that all of the component functions are convex, so $f_0$ is convex.  This allows us to compare our method with other shape restricted methods in the Gaussian setting.  Problem 3 represents a more challenging (higher-dimensional) problem.  In the Gaussian setting, we compared the performance of SCMLE with Shape Constrained Additive Models (SCAM) \citep{PyaWood2014}, Generalised Additive Models with Integrated Smoothness estimation (GAMIS) \citep{Wood2004}, Multivariate Adaptive Regression Splines with maximum interaction degree equal to one (MARS) \citep{Friedman1991}, regression trees \citep{BFOS1984}, Convex Adaptive Partitioning (CAP) \citep{HannahDunson2013}, and  Multivariate Convex Regression (MCR) \citep{LimGlynn2012,SeijoSen2011}. Some of the above-mentioned methods are not suitable to deal with non-identity link functions, so in the Poisson and Binomial settings, we only compared SCMLE with SCAM and GAMIS.  

SCAM can be viewed as a shape-restricted version of GAMIS. It is implemented in the \texttt{R} package \texttt{scam} \citep{Pya2012}. GAMIS is implemented in the \texttt{R} package \texttt{mgcv} \citep{Wood2012}, while MARS can be founded in the \texttt{R} package \texttt{mda} \citep{HTLHR2011}. The method of regression trees is implemented in the \texttt{R} package \texttt{tree} \citep{Ripley2012}, and CAP is implemented in MATLAB by \citet{HannahDunson2013}. We implemented MCR in MATLAB using the \texttt{interior-point-convex} solver. Default settings were used for all of the competitors mentioned above.

For different sample sizes $n$ = 200, 500, 1000, 2000, 5000, we ran all the methods on 50 randomly generated datasets. Our numerical experiments were carried out on standard 32-bit desktops with 1.8~GHz CPUs. Each method was given at most one hour per dataset. Beyond this limit, the run was forced to stop and the corresponding results were omitted. Tables~\ref{Tab:simtimeG} and~\ref{Tab:simtimePB} in the online supplementary material provide the average running time of different methods per training dataset. Unsurprisingly SCMLE is slower than Tree or MARS, particularly in the higher-dimensional setting.  On the other hand, it is typically faster than other shape-constrained methods such as SCAM and MCR.  Note that MCR is particularly slow compared to the other methods, and becomes computationally infeasible for $n \ge 1000$.

\begin{table}[!p]\small
  \centering
  \begin{tabular}{|c|c|c|c|c|c|}
  \multicolumn{6}{c}{\normalsize Estimated MISEs: Gaussian} \\ 
  \multicolumn{6}{c}{Problem 1} \\ \hline
  Method & $n=200$ & $n=500$ & $n=1000$ & $n=2000$ & $n=5000$ \\ \hline
  SCMLE   & 0.413 & \textbf{0.167} & \textbf{0.085} & \textbf{0.044} & \textbf{0.021}  \\
  SCAM    & \textbf{0.406} & 0.247 & 0.162 & 0.133 & 0.079  \\
  GAMIS   & 0.412 & 0.177 & 0.095 & 0.049 & 0.024  \\
  MARS    & 0.538 & 0.249 & 0.135 & 0.087 & 0.044  \\
  Tree    & 3.692 & 2.805 & 2.488 & 2.345 & 2.342  \\
  CAP     & 3.227 & 1.689 & 0.912 & 0.545 & 0.280  \\
  MCR     & 203.675 & 8415.607 & - & - & -  \\\hline
  \multicolumn{6}{c}{Problem 2} \\ \hline
  Method & $n=200$ & $n=500$ & $n=1000$ & $n=2000$ & $n=5000$ \\ \hline
  SCMLE   & 0.273 & \textbf{0.101} & \textbf{0.053} & \textbf{0.028} & \textbf{0.012} \\
  SCAM    & \textbf{0.264} & 0.107 & 0.058 & 0.032 & 0.016 \\
  GAMIS   & 0.363 & 0.154 & 0.079 & 0.041 & 0.019 \\
  MARS    & 0.417 & 0.177 & 0.087 & 0.050 & 0.021 \\
  Tree    & 1.995 & 1.277 & 1.108 & 1.015 & 0.973 \\
  CAP     & 1.282 & 0.742 & 0.415 & 0.251 & 0.145 \\
  MCR     & 940.002 & 14557.570 & - & - & -  \\\hline
  \multicolumn{6}{c}{Problem 3} \\ \hline
  Method & $n=200$ & $n=500$ & $n=1000$ & $n=2000$ & $n=5000$ \\ \hline
  SCMLE   & 11.023 & \textbf{3.825} & \textbf{2.096} & \textbf{1.107} & \textbf{0.479} \\
  SCAM    & \textbf{9.258} & 4.931 & 3.659 & 2.730 & 2.409 \\
  GAMIS   & 11.409  & 4.578  & 2.498  & 1.398  & 0.630 \\
  MARS    & 14.618  & 6.614  & 4.940  & 3.580  & 3.056 \\
  Tree    & 120.310 & 94.109 & 87.118 & 80.846 & 80.388 \\
  CAP     & 92.846  & 72.308 & 50.964 & 38.615 & 29.577 \\
  MCR     & 107.547 & 1535.022 & - & - & -  \\\hline
  \end{tabular}
  \caption{Estimated MISEs in the Gaussian setting for Problems~1, 2 and 3. The lowest MISE values are in bold font.} 
  \label{Tab:simmiseG}
\end{table}

\begin{table}[!p]\small
  \centering
  \begin{tabular}{|c|c|c|c|c|c|c|}
  \multicolumn{7}{c}{\normalsize Estimated MISEs: Poisson and Binomial} \\ 
  \multicolumn{7}{c}{Problem 1} \\ \hline
  Model & Method & $n=200$ & $n=500$ & $n=1000$ & $n=2000$ & $n=5000$ \\ \hline
           & SCMLE   & 0.344 & \textbf{0.131} & \textbf{0.067} & \textbf{0.038} & \textbf{0.017} \\
  Poisson  & SCAM    & 0.342 & 0.211 & 0.135 & 0.106 & 0.069 \\
           & GAMIS   & \textbf{0.330} & 0.142 & 0.078 & 0.043 & 0.021 \\\hline
           & SCMLE   & 0.933 & \textbf{0.282} & \textbf{0.146} & \textbf{0.079} & \textbf{0.037} \\
  Binomial & SCAM    & \textbf{0.500} & 0.324 & 0.271 & 0.241 & 0.222 \\
           & GAMIS   & 0.639 & 0.284 & 0.153 & 0.085 & 0.040 \\\hline
  \multicolumn{7}{c}{Problem 2} \\ \hline
  Model & Method & $n=200$ & $n=500$ & $n=1000$ & $n=2000$ & $n=5000$ \\ \hline
           & SCMLE   & 0.439 & \textbf{0.139} & \textbf{0.079} & \textbf{0.042} & \textbf{0.019}  \\
  Poisson  & SCAM    & \textbf{0.384} & 0.184 & 0.092 & 0.047 & 0.024  \\
           & GAMIS   & 0.505 & 0.210 & 0.121 & 0.064 & 0.030  \\\hline
           & SCMLE   & \textbf{0.357} & \textbf{0.132} & \textbf{0.065} & \textbf{0.036} & \textbf{0.016}  \\
  Binomial & SCAM    & 0.453 & 0.227 & 0.137 & 0.072 & 0.025  \\
           & GAMIS   & 0.449 & 0.173 & 0.090 & 0.054 & 0.024  \\\hline
  \multicolumn{7}{c}{Problem 3} \\ \hline
  Model & Method & $n=200$ & $n=500$ & $n=1000$ & $n=2000$ & $n=5000$ \\ \hline
           & SCMLE   & \textbf{4.399} & \textbf{1.509} &  \textbf{0.748} & \textbf{0.408} & \textbf{0.181}  \\
  Poisson  & SCAM    &  5.415 & 3.362 &  2.544 & 2.085 & 1.705  \\
           & GAMIS   &  4.698 & 1.949 &  0.981 & 0.571 & 0.275 \\\hline
           & SCMLE   & 40.614 & \textbf{11.343} &  \textbf{5.694} &  \textbf{2.973} &  \textbf{1.291}  \\
  Binomial & SCAM    & \textbf{23.801} & 16.505 & 13.992 & 12.868 & 12.207  \\
           & GAMIS   & 25.439 & 11.908 &  6.317 &  3.516 &  1.551  \\\hline
    \end{tabular}
  \caption{Estimated MISEs in the Poisson and Binomial settings for Problems~1, 2 and 3. The lowest MISE values are in bold font.}
  \label{Tab:simmisePB}
\end{table}
To study the empirical performance of SCMLE, we drew $10^5$ covariates independently from $U[-0.98,0.98]^d$ and estimated the Mean Integrated Squared Error (MISE) $\mathbb{E}\int_{[-0.98,0.98]^d}(\hat{f}_n - f_0)^2$ using Monte Carlo integration.  Estimated MISEs are given in Tables~\ref{Tab:simmiseG} and~\ref{Tab:simmisePB}.  For every setting we considered, SCMLE performs better than Tree, CAP and MCR. This is largely due to the fact that these three estimators do not take into account the additive structure.  In particular, MCR suffers severely from its boundary behaviour. It is also interesting to note that for small $n=200$, SCAM and GAMIS occasionally offer slightly better performance than SCMLE. This is also mainly caused by the boundary behaviour of SCMLE, and is alleviated as the number of observations $n$ increases.  In each of the three problems considered, SCMLE enjoys better predictive performance than the other methods for $n \ge 500$.  SCMLE appears to offer particular advantages when the true signal exhibits inhomogeneous smoothness, since it is able to regularise in a locally adaptive way, while both SCAM and GAMIS rely on a single level of regularisation throughout the covariate space.

\subsection{Generalised additive index models with shape restrictions}

In our comparisons of different estimators in GAIMs, we focused on the Gaussian case to facilitate comparisons with other methods. We took $\bX_1,\ldots,\bX_n \stackrel{\mathrm{i.i.d.}}{\sim} U[-1,1]^d$, and considered the following two problems:
\begin{enumerate}[1.]
\setlength{\itemsep}{0pt}
\setlength{\parskip}{0pt}
\setlength{\parsep}{0pt}
\setcounter{enumi}{3}
\item Here $d = 4$ and $m = 1$. We set $L_1=4$ and $ f_0^I(\bx) = |0.25 x_1 + 0.25 x_2 + 0.25 x_3 + 0.25 x_4|$.

\item Here $d = 2$ and $m = 2$. We set $\bL_2=(4,7)^T$ and $f_0^I(\bx)  = (0.5 x_1 + 0.5 x_2)^2 - |0.5 x_1 - 0.5 x_2|^3$.
\end{enumerate}
In both problems, conditional on $\bX_i$, we drew independently $Y_i \sim N(f_0^I(\bX_i),0.5^2)$ for $i = 1, \ldots, n$.  We compared the performance of our SCAIE with Projection Pursuit Regression (PPR) \citep{FriedmanStuetzle1981}, Multivariate Adaptive Regression Splines with maximum two interaction degrees (MARS) and regression trees (Tree). In addition, in Problem~4, we also considered the Semiparametric Single Index (SSI) method \citep{Ichimura1993}, CAP and MCR.  SSI was implemented in the \texttt{R} package \texttt{np} \citep{HayfieldRacine2013}.   SCAIE was computed using the algorithm illustrated in Table~\ref{Tab:codeSCAIE}. We picked the total number of stochastic searches to be $N=100$. Note that because Problem~4 is a single-index problem (i.e. $m=1$), there is no need to supply $\delta$. In Problem~5, we chose $\delta = 0.1$. We considered sample sizes $n$ = 200, 500, 1000, 2000, 5000. 

Table~\ref{Tab:simtimeindex} in the online supplementary material gives the average running time of different methods per training dataset. Although SCAIE is slower than PPR, MARS and Tree, its computation can be accomplished within a reasonable amount of time even when $n$ is as large as 5000. As SSI adopts a leave-one-out cross-validation strategy, it is typically considerably slower than SCAIE. 

Estimated MISEs of different estimators over $[-0.98,0.98]^d$ are given in Table~\ref{Tab:simmiseindex}.  In both Problems 4 and 5, we see that SCAIE outperforms its competitors for all the sample sizes we considered.  It should, of course, be noted that SSI, PPR, MARS and Tree do not enforce the shape constraints, while MARS, Tree, CAP and MCR do not take into account the additive index structure.  

\begin{table}[!p]\small
  \centering
  \begin{tabular}{|c|c|c|c|c|c|}
  \multicolumn{6}{c}{\normalsize Estimated MISEs: Additive Index Models} \\ 
  \multicolumn{6}{c}{Problem 4} \\ \hline
  Method & $n=200$ & $n=500$ & $n=1000$ & $n=2000$ & $n=5000$ \\ \hline
  SCAIE   & \textbf{0.259} & \textbf{0.074} & \textbf{0.038} & \textbf{0.019} & \textbf{0.008} \\
  SSI     & 0.878 & 0.478 & 0.313 & 0.206 &     - \\
  PPR	  & 0.679 & 0.419 & 0.277 & 0.201 & 0.151 \\
  MARS    & 0.634 & 0.440 & 0.238 & 0.179 & 0.143 \\
  Tree    & 1.903 & 0.735 & 0.425 & 0.409 & 0.406 \\
  CAP     & 0.348 & 0.137 & 0.081 & 0.056 & 0.016 \\
  MCR     & 2539.912 & 35035.710 & - & - & - \\ \hline
  \multicolumn{6}{c}{Problem 5} \\ \hline
  Method & $n=200$ & $n=500$ & $n=1000$ & $n=2000$ & $n=5000$ \\ \hline
  SCAIE   & \textbf{0.078} & \textbf{0.030} & \textbf{0.016} & \textbf{0.008} & \textbf{0.005}   \\
  PPR	  & 0.137 & 0.055 & 0.027 & 0.015 & 0.010   \\
  MARS    & 0.081 & 0.034 & 0.018 & 0.010 & 0.006   \\
  Tree    & 0.366 & 0.241 & 0.266 & 0.310 & 0.309   \\\hline
  \end{tabular}
  \caption{\small Estimated MISEs in Problems 4 and 5. The lowest MISE values are in bold font.}
  \label{Tab:simmiseindex}
\end{table}

\begin{table}[!htbp]\small
  \centering
  \begin{tabular}{|c|c|c|c|c|c|}
  \multicolumn{6}{c}{Problem 4: estimated RMSEs} \\ \hline
  Method & $n=200$ & $n=500$ & $n=1000$ & $n=2000$ & $n=5000$ \\ \hline
  SCAIE   & \textbf{0.230}  & \textbf{0.100}  & \textbf{0.056}  & \textbf{0.038} & \textbf{0.024}  \\
  SSI     & 0.677  & 0.615  & 0.595  & 0.492 & - \\
  PPR	  & 0.583  & 0.596  & 0.539  & 0.481 & 0.454\\\hline
  \multicolumn{6}{c}{Problem 5: estimated Amari distance} \\ \hline
  Method & $n=200$ & $n=500$ & $n=1000$ & $n=2000$ & $n=5000$ \\ \hline
  SCAIE   & \textbf{0.216}  & \textbf{0.135}  & \textbf{0.090}  & \textbf{0.062} &  \textbf{0.045} \\ 
  PPR     & 0.260  & 0.214  & 0.144  & 0.104 &  0.067 \\\hline
  \end{tabular}
  \caption{\small Distance between the estimated index matrix and the truth: RMSEs were estimated in Problem~4, while the mean Amari errors were estimated in Problem~5. The lowest distances are in bold font.}
  \label{Tab:simindexdist}
\end{table}
In the index setting, it is also of interest to compare the performance of those methods that directly estimate the index matrix.  We therefore estimated Root Mean Squared Errors (RMSEs), given by $\sqrt{\mathbb{E} \| \tilde{\balpha}_{n,1} - \balpha_{0,1} \|_2^2}$ in Problem 4, where $\balpha_{0,1} = (0.25,0.25,0.25,0.25)^T$.  For Problem~5, we estimated mean errors in Amari distance $\rho$, defined by \citet{ACY1996} as 
\[
\rho(\tilde{\bA}_n,\bA_0) = \frac{1}{2d}\sum_{i=1}^{d} \bigg( \frac{\sum_{j=1}^d |C_{ij}|}{ \max_{1\le j \le d} |C_{ij}|} - 1 \bigg) + \frac{1}{2d}\sum_{j=1}^{d} \bigg( \frac{\sum_{i=1}^d |C_{ij}|}{ \max_{1\le i \le d} |C_{ij}|} - 1 \bigg),
\]
where $C_{ij} = (\tilde{\bA}_n \bA_0^{-1})_{ij}$ and $\bA_0 = \begin{pmatrix} 0.5 & 0.5 \\ 0.5 &-0.5 \end{pmatrix}$.  This distance measure is invariant to permutation and takes values in $[0,d-1]$.  Results obtained for SCAIE and, where applicable, SSI and PPR, are displayed in Table~\ref{Tab:simindexdist}.   For both problems, SCAIE performs better in these senses than both SSI and PPR in terms of estimating the projection indices.

\section{Real data examples}
\label{Sec:real}

In this section, we apply our estimators in two real data examples. In the first, we study doctoral publications in biochemistry and fit a generalised (Poisson) additive model with concavity constraints; while in the second, we use an additive index model with monotonicity constraints to study javelin performance in the decathlon.

\subsection{Doctoral publications in biochemistry}
\label{Sec:publication}
The scientific productivity of a doctoral student may depend on many factors, including some or all of the number of young children they have, the productivity of the supervisor, their gender and marital status. \citet{Long1990} studied this topic focusing on the gender difference; see also \citet{Long1997}.  The dataset is available in the \texttt{R} package \texttt{AER} \citep{KleiberZeileis2013}, and contains $n=915$ observations.  Here we model the number of articles written by the $i$-th PhD student in the last three years of their PhD as a Poisson random variable with mean $\mu_i$, where 
\[
 \log \mu_i =  f_1(\texttt{kids}_i) + f_2(\texttt{mentor}_i) + a_3 \, \texttt{gender}_i + a_4 \, \texttt{married}_i + c,  
\]
for $i = 1,\ldots,n$, where $\texttt{kids}_i$ and $\texttt{mentor}_i$ are respectively the number of that student's children that are less than 6 years old, and the number of papers published by that student's supervisor during the same period of time.  Both $\texttt{gender}_i$ and $\texttt{married}_i$ are factors taking values 0 and 1, where 1 indicates `female' and `married' respectively.  In the original dataset, there is an extra continuous variable that measures the prestige of the graduate program.  We chose to drop this variable in our example because: (i) its values were determined quite subjectively; and (ii) including this variable does not seem to improve the predictive power in the above settings.

To apply SCMLE, we assume that $f_1$ is a concave and monotone decreasing function, while $f_2$ is a concave function.  The main estimates obtained from SCMLE are summarised in Table~\ref{Tab:publication1} and Figure~\ref{Fig:publication}.  Outputs from SCAM and GAMIS are also reported for comparison.  We see that with the exception of $\hat{f}_{n,2}$, estimates obtained from these methods are relatively close.  Note that in Figure~\ref{Fig:publication}, the GAMIS estimate of $f_2$ displays local fluctuations that might be harder to interpret than the estimates obtained using SCMLE and SCAM. 
\begin{table}[!htbp]\small
  \centering
  \begin{tabular}{|c|cccc||c c |} \hline
  Method & $\hat{f}_{n,1}(0)$ & $\hat{f}_{n,1}(1)$ & $\hat{f}_{n,1}(2)$ & $\hat{f}_{n,1}(3)$ & $\hat{a}_{n,3}$ & $\hat{a}_{n,4}$ \\ \hline
  SCMLE  & 0 & -0.110 & -0.284 & -0.816 & -0.218 & 0.126 \\
  SCAM   & 0 & -0.136 & -0.303 & -0.770 & -0.224 & 0.152 \\
  GAMIS  & 0 & -0.134 & -0.301 & -0.784 & -0.226 & 0.157 \\ \hline
  \end{tabular}
  \caption{\small Estimates obtained from SCMLE, SCAM and GAMIS on the PhD publication dataset.}
  \label{Tab:publication1}
\end{table}

Finally, we examine the prediction power of the different methods via cross-validation.  Here we randomly split the dataset into training (70\%) and validation (30\%) subsets. For each split, we compute estimates using only the training set, and assess their predictive accuracy in terms of Root Mean Square Prediction Error (RMSPE) on the validation set. The reported RMSPEs in Table~\ref{Tab:publication2} are averages over 500 splits.  Our findings suggest that whilst comparable to SCAM, SCMLE offers slight improvements over GAMIS and Tree for this dataset.
\begin{figure}[!htbp]
  \centering
  \includegraphics[height=13cm,angle=270]{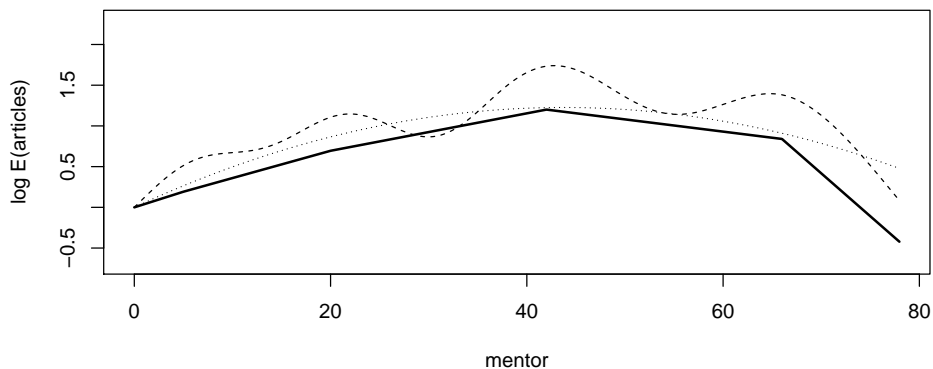}
  \caption{\small Different estimates of $f_2$: SCMLE (solid), SCAM (dotted) and GAMIS (dashed).}
  \label{Fig:publication}
\end{figure}

\begin{table}[!htbp]\small
  \centering
  \begin{tabular}{| c | c c c c|} \hline
  Method & SCMLE & SCAM & GAMIS & Tree\\ \hline
  RMSPE  & \textbf{1.822} & 1.823 & 1.838 & 1.890 \\ \hline
  \end{tabular}
  \caption{\small Estimated prediction errors of SCMLE, SCAM, GAMIS and Tree on the PhD publication dataset. The smallest RMSPE is in bold font.}
  \label{Tab:publication2}
\end{table}

\subsection{Javelin throw}
\label{Sec:Javelin}
In this section, we consider the problem of predicting a decathlete's javelin performance from their performances in the other decathlon disciplines.  Our dataset consists of decathlon athletes who scored at least 6500 points in at least one athletic competition in 2012 and scored points in every event there.  To avoid data dependency, we include only one performance from each athlete, namely their 2012 personal best (over the whole decathlon).  The dataset, which consists of $n=614$ observations, is available in the \texttt{R} package \texttt{scar} \citep{ChenSamworth2014}. For simplicity, we only select events (apart from Javelin) that directly reflect the athlete's ability in throwing and short-distance running, namely, shot put, discus, 100 metres and 110 metres hurdles. We fit the following additive index model:
\begin{align*}
	\texttt{javelin}_i \,= \,& f_1(A_{11}\,\texttt{100m}_i + A_{21}\,\texttt{110m}_i + A_{31}\,\texttt{shot}_i + A_{41}\,\texttt{discus}_i) \\
			     & + f_2(A_{12}\,\texttt{100m}_i + A_{22}\,\texttt{110m}_i + A_{32}\,\texttt{shot}_i + A_{42}\,\texttt{discus}_i) + \epsilon_i, 
\end{align*}
for $i = 1,\ldots,614$, where $\epsilon_i \stackrel{\mathrm{i.i.d.}}{\sim} N(0,\sigma^2)$, and where $\texttt{javelin}_i$, $\texttt{100m}_i$, $\texttt{110m}_i$, $\texttt{shot}_i$ and $\texttt{discus}_i$ represent the corresponding decathlon event scores for the $i$-th athlete.   For SCAIE, we assume that both $f_1$ and $f_2$ are monotone increasing, and also assume that $A_{11},\ldots,A_{41},A_{12},\ldots,A_{42}$ are non-negative.  This slightly restricted version of SCAIE aids interpretability of the indices, and prevents `perfect-fit' phenomenon (cf. Section~\ref{Sec:gaimMLE}), so no choice of $\delta$ is required.

Table~\ref{Tab:javelin1} gives the estimated index loadings by SCAIE.  We observe that the first projection index can be interpreted as the general athleticism associated with the athlete, while the second can be viewed as a measure of throwing ability.  Note that, when using SCAIE, $\hat{A}_{n,12}$ and $\hat{A}_{n,22}$ are relatively small.  To further simplify our model, and to seek improvement in the prediction power, we therefore considered forcing these entries to be exactly zero in the optimisation steps of SCAIE.  This sparse version is denoted as $\mathrm{SCAIE}_s$. Its estimated index loadings are also reported in Table~\ref{Tab:javelin1}.

\begin{table}[!htbp]\small
  \centering
  \begin{tabular}{|l | c c c c || c c c c |} \hline
  Method & $\hat{A}_{n,11}$ & $\hat{A}_{n,21}$ & $\hat{A}_{n,31}$ & $\hat{A}_{n,41}$ & $\hat{A}_{n,12}$ & $\hat{A}_{n,22}$ & $\hat{A}_{n,32}$ & $\hat{A}_{n,42}$ \\\hline
  SCAIE			& 0.262 & 0.343 & 0.222 & 0.173 & 0.006 & 0.015 & 0.522 & 0.457\\ 
  $\mathrm{SCAIE}_s$	& 0.235 & 0.305 & 0.140 & 0.320 & 0     & 0     & 0.536 & 0.464\\ \hline
  \end{tabular}
  \caption{\small Estimated index loadings by SCAIE and $\mathrm{SCAIE}_s$}
  \label{Tab:javelin1}
\end{table}

To compare the performance of our methods with PPR, MARS with maximum two degrees of interaction and Tree, we again estimated the prediction power (in terms of RMSPE) via 500 repetitions of 70\%/30\% random splits into training/test sets. The corresponding RMSPEs are reported in Table~\ref{Tab:javelin2}. We see that both SCAIE and $\mathrm{SCAIE}_s$ outperform their competitors in this particular dataset. It is also interesting to note that $\mathrm{SCAIE}_s$ has a slightly lower RMSPE than SCAIE, suggesting that the simpler (sparser) model might be preferred for prediction here. 

\begin{table}[!htbp]\small
  \centering
  \begin{tabular}{|c | c c c c c|} \hline
  Method & SCAIE & $\mathrm{SCAIE}_s$ & PPR & MARS & Tree\\ \hline
  RMSPE  & 81.276 & \textbf{80.976} & 82.898 & 82.915 & 85.085 \\ \hline
  \end{tabular}
  \caption{\small Estimated prediction errors of SCAIE, $\mathrm{SCAIE}_s$, PPR, MARS and Tree on the decathlon dataset. The smallest RMSPE is in bold font.}
  \label{Tab:javelin2}
\end{table}

\section*{Acknowledgements} We would like to thank Mary Meyer for providing us with her manuscript \citet{Meyer2013a} prior to its publication.  Both authors are supported by the second author's Engineering and Physical Sciences Research Fellowship EP/J017213/1.

\section{Appendix}
\begin{prooftitle}{of Proposition~\ref{Prop:gamexistunique}}
Define the set
\[
	\Theta = \{\boldsymbol{\eta} = (\eta_1,\ldots,\eta_n)^T \in \bar{\mathbb{R}}^n \ | \ \exists f \in \mathrm{cl}(\mathcal{F}^{\bL_d}) \mbox{ s.t. } \eta_i = f(\bX_i), \forall i = 1,\ldots,n\}.
\]
We can rewrite the optimisation problem as finding $\hat{\boldsymbol{\eta}}_n$ such that 
\[
	\hat{\boldsymbol{\eta}}_n \in \argmax_{\boldsymbol{\eta} \in \Theta} \bar{\ell}_n(\boldsymbol{\eta}),
\]
where $\bar{\ell}_n(\boldsymbol{\eta})= \frac{1}{n} \sum_{i=1}^n \ell_i(\eta_i)$, and where
\[
\ell_i(\eta_i) = \begin{cases}
  Y_i \eta_i - B(\eta_i), 
  &  \text{if } \eta_i \in \mathrm{dom}(B); \\
  \lim_{a \rightarrow -\infty}Y_i a - B(a), & \text{if } \eta_i = -\infty; \\
  \lim_{a \rightarrow \infty} Y_i a - B(a), & \text{if the EF is Gaussian, Poisson or Binomial, and } \eta_i = \infty; \\
  -\infty, & \text{if the EF is Gamma and } \eta_i \in [0,\infty].
\end{cases}
\]
Note that $\bar{\ell}_n$ is continuous on the non-empty set $\Theta$ and $\sup_{\boldsymbol{\eta} \in \Theta} \bar{\ell}_n(\boldsymbol{\eta})$ is finite.  Moreover, by Lemma~\ref{Lem:Closed} in the online supplementary material, $\Theta$ is a closed subset of the compact set $\bar{\mathbb{R}}^n$, so is compact.  It follows that $\bar{\ell}_n$ attains its maximum on $\Theta$, so $\hat{S}_n^{\bL_d} \neq \emptyset$.

To show the uniqueness of $\hat{\boldsymbol{\eta}}_n$, we now suppose that both $\boldsymbol{\eta}_1 = (\eta_{11},\ldots,\eta_{1n})^T$ and $\boldsymbol{\eta}_2 = (\eta_{21},\ldots,\eta_{2n})^T$ maximise $\bar{\ell}_n$.  The only way we can have $\eta_{1i} = \infty$ is if the family is Binomial and $Y_i = 1$.  But then $\ell_i(-\infty) = -\infty$, so we cannot have $\eta_{2i} = -\infty$.  It follows that $\boldsymbol{\eta}_* = (\boldsymbol{\eta}_1 + \boldsymbol{\eta}_2 )/2$ is well-defined, and $\boldsymbol{\eta}_* \in \Theta$, since $\Theta$ is convex.  Now we can use the strict concavity of $\bar{l}_n$ on its domain to conclude that $\boldsymbol{\eta}_1 = \boldsymbol{\eta}_2 = \boldsymbol{\eta}_*$.
\end{prooftitle}

To prove Theorem~\ref{Thm:SCMLEconsistency}, we require the following lemma, which says (roughly) that if any of the additive components (or the intercept) of $f \in \mathcal{F}^{\bL_d}$ are large somewhere, then there is a non-trivial region on which either $f$ is large, or a region on which $-f$ is large.
\begin{lem}
\label{Lem:AddFuncBound}
Fix $a > 0$.  There exists a finite collection $\mathcal{C}_a$ of disjoint compact subsets of $[-2a,2a]^d$ each having Lebesgue measure at least $\bigl(\frac{a}{2d}\bigr)^d$, such that for any $f \stackrel{\mathcal{F}^{\bL_d}}{\sim} (f_1,\ldots,f_d,c)$,
\[
\max_{C \in \mathcal{C}_a} \max\Bigl\{\inf_{\bx \in C} f(\bx),\inf_{\bx \in C} -f(\bx)\Bigr\} \geq \frac{1}{4}\max \biggl\{\sup_{|x_1|\le a} |f_1(x_1)|, \ldots, \sup_{|x_d|\le a} |f_d(x_d)|, 2 |c|\biggr\}.
\]
\end{lem}
\begin{proof}
Let $\max \bigl\{\sup_{|x_1|\le a} |f_1(x_1)|, \ldots, \sup_{|x_d|\le a} |f_d(x_d)|, 2 |c|\bigr \} = M$ for some $M \ge 0$.  Recalling that $f_1(0) = \ldots = f_d(0) = 0$, and owing to the shape restrictions, this is equivalent to
\[
	\max \bigl\{|f_1(-a)|, |f_1(a)|, \ldots, |f_d(-a)|, |f_d(a)|, 2 |c|\bigr \} = M.
\]
We will prove the lemma by construction. For $j = 1,\ldots,d$, consider the collection of intervals
\begin{align*}
\mathcal{D}_j = \begin{cases} \Big\{[-2a,-a],[a,2a] \Big\}, & \mbox{ if } l_j \in \{2,3,5,6,8,9\}; \\ 
\Big\{[-2a,-a],[-a/(4d), a/(4d)], [a,2a]\Big\}, & \mbox{ if } l_j \in \{1,4,7\}. \end{cases}
\end{align*}
Let $\mathcal{C}_a = \bigl\{\times_{j=1}^d D_j : D_j \in \mathcal{D}_j\bigr\}$, so that $|\mathcal{C}_a| \leq 3^d$.  The two cases below validate our construction:
\begin{enumerate}[1.]
\setlength{\itemsep}{0pt}
\setlength{\parskip}{0pt}
\setlength{\parsep}{0pt}
\item $\max \left\{|f_1(-a)|, |f_1(a)|, \ldots, |f_d(-a)|, |f_d(a)| \right \} < M$. Then it must be the case that $|c| = M/2$ and, without loss of generality, we may assume $c = M/2$. For $j = 1,\ldots,d$, if $l_j \in \{2,3,5,6,8,9\}$, then due to the monotonicity and the fact that $f_j(0)=0$, either
\[
	\inf_{x_j \in [-2a,-a]} f_j(x_j) \ge 0 \ \mbox{ or } \  \inf_{x_j \in [a,2a]} f_j(x_j) \ge 0. 
\]
For $l_j \in \{1,4,7\}$, by the convexity/concavity, $\sup_{x_j \in [-a/(4d), a/(4d)]} |f_j(x_j)| \le M/(4d)$.  Hence 
\[
	\max_{C \in \mathcal{C}_a} \inf_{\bx \in C}f(\bx) \ge - dM/(4d) + M/2 = M/4.
\]
\item $\max \left\{|f_1(-a)|, |f_1(a)|, \ldots, |f_d(-a)|, |f_d(a)| \right \} = M$ and $|c| \le M/2$. Without loss of generality, we may assume that $f_1(-a) = M$. Since $f_1(0)=0$ and $|f_1(a)| \le M$, we can assume $l_1 \in \{1,3,4,6\}$. Therefore, $\inf_{x_1 \in [-2a,-a]} f_1(x_1) = M$. Let 
\[
D_j = \begin{cases} [-a/(4d),a/(4d)], & \mbox{ if } l_j \in \{1,4,7\} \\ [a,2a], & \mbox{ if } l_j \in \{2,5,8\} \\  [-2a,-a], & \mbox{ if } l_j \in \{3,6,9\} \end{cases}
\]
for $j = 2,\ldots,d$. Now for $C =[-2a,-a] \times \times_{j=2}^d D_j$, we have  
\[
\inf_{\bx \in C}f(\bx) \ge M - (d-1)M/(4d) - M/2 \ge M/4.
\]
\end{enumerate}
\end{proof}
\bigskip

\begin{prooftitle}{of Theorem~\ref{Thm:SCMLEconsistency}}
For convenience, we first present the proof of consistency in the case where the EF distribution is Binomial.  Consistency for the other EF distributions listed in Table~\ref{Tab:glmformula} can be established using essentially the same proof structure with some minor modifications. We briefly outline these changes at the end of the proof.  Our proof can be divided into five steps.

\textbf{Step 1: Lower bound for the scaled partial log-likelihood}. It follows from Assumption \textbf{(A.1)} and the strong law of large numbers that
\[
\liminf_{n \rightarrow \infty} \sup_{f \in \mathrm{cl}(\mathcal{F}^{\bL_d})} \bar{\ell}_n(f) \ge  \lim_{n \rightarrow \infty} \bar{\ell}_n(f_0) = \mathbb{E} \big\{ g^{-1}(f_0(\bX))f_0(\bX) - B(f_0(\bX)) \big\} =: \bar{L}_0
\]
almost surely. 

\textbf{Step 2: Bounding $|\hat{f}_n|$ on $[-a,a]^d$} for any fixed $a > 0$.  For $M > 0$, let
\begin{align}
\label{Eq:setFLaM}
\mathcal{F}_{a,M}^{\bL_d} = \bigg\{ f \stackrel{\mathcal{F}^{\bL_d}}{\sim} (f_1,\ldots,f_d,c): \max \{|f_1(-a)|, |f_1(a)|, \ldots, |f_d(-a)|, |f_d(a)|, 2 |c| \} \le M \bigg\}.
\end{align}
We will prove that there exists a deterministic constant $M = M(a) \in (0,\infty)$ such that, with probability one, we have $\hat{S}_n^{\bL_d} \subseteq \mathrm{cl}\big(\mathcal{F}_{a,M(a)}^{\bL_d}\big)$ for sufficiently large $n$.  To this end, let $\mathcal{C}_a = \{C_1,\ldots,C_N\}$ be the finite collection of compact subsets of $[-2a,2a]^d$ constructed in the proof of Lemma~\ref{Lem:AddFuncBound}, and set 
\[
	M = 4 B^{-1}\biggl(\frac{-\bar{L}_0 + 1}{\min_{1 \le k \le N , t \in \{0,1\}} \mathbb{P}\bigl(\bX \in C_{k}, Y = t\bigr)} \biggr).
\]
Note that
\begin{align}
\notag	&\limsup_{n \rightarrow \infty} \sup_{f \in \mathrm{cl}\big(\mathcal{F}^{\bL_d} \setminus \mathcal{F}_{a,M}^{\bL_d}\big)} \bar{\ell}_n(f) \\
\label{Eq:Gconsistency1}	& \le \max_{1 \le k \le N} \limsup_{n \rightarrow \infty} \sup_{f \in \mathrm{cl}\big(\mathcal{F}^{\bL_d} \backslash \mathcal{F}_{a,M}^{\bL_d}\big)} \frac{1}{n} \sum_{i = 1}^n \{ Y_i f(\bX_i)- B(f(\bX_i))\} \mathbbm{1}_{\{\bX_i \notin C_k \, \cup \, Y_i \notin \{0,1\}\}} \\
\label{Eq:Gconsistency2}	& \quad + \min_{1 \le k \le N} \limsup_{n \rightarrow \infty} \sup_{f \in \mathrm{cl}\big(\mathcal{F}^{\bL_d} \backslash \mathcal{F}_{a,M}^{\bL_d}\big)} \frac{1}{n} \sum_{i = 1}^n \{ Y_i f(\bX_i)- B(f(\bX_i))\} \mathbbm{1}_{\{\bX_i \in C_k \, \cap \, Y_i \in \{0,1\}\}}.
\end{align}
Now (\ref{Eq:Gconsistency1}) is non-positive, since $Y_i \eta - B(\eta) = Y_i \eta - \log(1+e^\eta) \leq 0$ for all $\eta \in \bar{\mathbb{R}}$ and $Y_i \in \{0,1/T,2/T,\ldots,1\}$.  We now claim that the supremum over $f \in \mathrm{cl}\big(\mathcal{F}^{\bL_d} \backslash \mathcal{F}_{a,M}^{\bL_d}\big)$ in~(\ref{Eq:Gconsistency2}) can be replaced with a supremum over $f \in \mathcal{F}^{\bL_d} \backslash \mathcal{F}_{a,M}^{\bL_d}$.  To see this, let
\[
\Theta_0 = \{\boldsymbol{\eta} = (\eta_1,\ldots,\eta_n)^T \in \bar{\mathbb{R}}^n: \exists f \in \mathrm{cl}(\mathcal{F}^{\bL_d} \setminus \mathcal{F}_{a,M}^{\bL_d}) \mbox{ s.t. } \eta_i = f(\bX_i), \forall i = 1,\ldots,n\}.
\]
Suppose that $(\boldsymbol{\eta}^m) \in \Theta_0$ is such that the corresponding $(f^m) \in \mathrm{cl}(\mathcal{F}^{\bL_d} \setminus \mathcal{F}_{a,M}^{\bL_d})$ is a maximising sequence in the sense that
\begin{align*}
\frac{1}{n} \sum_{i = 1}^n \{ Y_i f^m(\bX_i)- B(f^m(\bX_i))\} &\mathbbm{1}_{\{\bX_i \in C_k \, \cap \, Y_i \in \{0,1\}\}} \\
&\nearrow \sup_{f \in \mathrm{cl}(\mathcal{F}^{\bL_d} \setminus \mathcal{F}_{a,M}^{\bL_d})} \frac{1}{n} \sum_{i = 1}^n \{ Y_i f(\bX_i)- B(f(\bX_i))\} \mathbbm{1}_{\{\bX_i \in C_k \, \cap \, Y_i \in \{0,1\}\}}.
\end{align*}
By reducing to a subsequence if necessary, we may assume $\boldsymbol{\eta}^m \rightarrow \boldsymbol{\eta}^0$, say, as $m \rightarrow \infty$, where $ \boldsymbol{\eta}^0 = (\eta_1^0,\ldots,\eta_n^0)^T \in \bar{\mathbb{R}}^n$.  Since, for each $m \in \mathbb{N}$, we can find a sequence $(f^{m,k})_k \in \mathcal{F}^{\bL_d} \setminus \mathcal{F}_{a,M}^{\bL_d}$ such that $f^{m,k} \rightarrow f^m$ pointwise in $\bar{\mathbb{R}}$ as $k \rightarrow \infty$, it follows that we can pick $k_m \in \mathbb{N}$ such that $f^{m,k_m}(\bX_i) \rightarrow \eta_i^0$ as $m \rightarrow \infty$, for all $i=1,\ldots,n$.  Moreover, $(\eta_1,\ldots,\eta_n) \mapsto \frac{1}{n} \sum_{i = 1}^n \{ Y_i \eta_i- B(\eta_i)\} \mathbbm{1}_{\{\bX_i \in C_k \, \cap \, Y_i \in \{0,1\}\}}$ is continuous on $\bar{\mathbb{R}}^n$, and we deduce that $(f^{m,k_m}) \in \mathcal{F}^{\bL_d} \setminus \mathcal{F}_{a,M}^{\bL_d}$ is also a maximising sequence, which establishes our claim.

Recall that by Lemma~\ref{Lem:AddFuncBound}, for any $f \in \mathcal{F}^{\bL_d} \backslash \mathcal{F}_{a,M}^{\bL_d}$, we can always find $C_{k^*} \in \mathcal{C}_a$ such that 
\[
\max\Bigl\{\inf_{\bx \in C_{k^*}} f(\bx),\inf_{\bx \in C_{k^*}} -f(\bx)\Bigr\} \ge M/4.
\]
Combining the non-positivity of~(\ref{Eq:Gconsistency1}) and our argument above removing the closure in~(\ref{Eq:Gconsistency2}), we deduce by the strong law of large numbers that 
\begin{align*}
\limsup_{n \rightarrow \infty} &\sup_{f \in \mathrm{cl}\big(\mathcal{F}^{\bL_d} \setminus \mathcal{F}_{a,M}^{\bL_d}\big)} \bar{\ell}_n(f) \\
&\leq \max\Bigl\{-B(M/4) \mathbb{P}(\bX \in C_{k^*},Y=0), \quad \{-M/4 - B(-M/4)\}\mathbb{P}(\bX \in C_{k^*},Y=1) \Bigl\} \\
&\leq -\min_{1\le k \le N, t \in \{0,1\}} \mathbb{P}\bigl(\bX \in C_{k}, Y = t\bigr)B(M/4) = \bar{L}_0-1,
\end{align*}
where, we have used the property that $B(t) = t + B(-t)$ for the penultimate inequality, and the definition of $M$ for the final equality.  Comparing this bound with the result of Step~1, we deduce that $\hat{S}_n^{\bL_d} \cap \mathrm{cl}\big(\mathcal{F}^{\bL_d} \setminus \mathcal{F}_{a,M}^{\bL_d}\big) = \emptyset$ for sufficiently large $n$, almost surely.  But it is straightforward to check that $\mathrm{cl}(\mathcal{F}^{\bL_d}) = \mathrm{cl}(\mathcal{F}_{a,M}^{\bL_d}) \cup \mathrm{cl}(\mathcal{F}^{\bL_d} \setminus \mathcal{F}_{a,M}^{\bL_d})$, and the result follows.

\textbf{Step 3: Lipschitz constant for the convex/concave components of $\hat{f}_n$ on $[-a,a]$}. For $M_1, M_2 >0$, let
\begin{align*}
	\mathcal{F}_{a,M_1,M_2}^{\bL_d} = \bigg\{ f \stackrel{\mathcal{F}^{\bL_d}}{\sim} (f_1,\ldots,f_d,c) \in \mathcal{F}_{a,M_1}^{\bL_d}: |f_j(z_1) - f_j(z_2)| &\le M_2 |z_1 - z _2|, \forall z_1,z_2 \in [-a,a], \\
	&\forall j \mbox{ with } l_j \in \{1,4,5,6,7,8,9\}\bigg\}.
\end{align*}
For notational convenience, we define $W(a) = M(a) + M(a+1) + 1$.  By Lemma~\ref{Lem:Intersect} in the online supplementary material,
\[
\mathrm{cl}\big(\mathcal{F}_{a,M(a)}^{\bL_d}\big) \cap \mathrm{cl}\big(\mathcal{F}_{a+1,M(a+1)}^{\bL_d}\big)  \subseteq \mathrm{cl}\big(\mathcal{F}_{a,M(a),W(a)}^{\bL_d}\big).
\]
From this and the result of Step~2, we have that for any fixed $a > 0$, with probability one, $\hat{S}_n^{\bL_d} \subseteq \mathrm{cl}\big(\mathcal{F}_{a,M(a),W(a)}^{\bL_d}\big)$ for sufficiently large $n$.

\textbf{Step 4: Glivenko--Cantelli Classes}.

For, $a>0$, $M_1 > 0$, $M_2 > 0$ and $j=1,\ldots,d$, let 
\[
\check{\mathcal{F}}_{a,M_1,M_2}^{\bL_d} = \Big\{ \check{f} : \mathbb{R}^d \rightarrow \mathbb{R} \Big| \check{f}(\bx) = f(\bx) \mathbbm{1}_{\{\bx \in [-a,a]^d\}}, f \in \mathcal{F}_{a,M_1,M_2}^{\bL_d} \Big\}
\]
and
\[
\big(\check{\mathcal{F}}_{a,M_1,M_2}^{\bL_d}\big)_j = \Big\{\check{f} : \mathbb{R}^d \rightarrow \mathbb{R}\Big| \check{f}(\bx) = f_j(x_j)\mathbbm{1}_{\{\bx \in [-a,a]^d\}} \text{ for some } f \stackrel{\mathcal{F}^{\bL_d}}{\sim} (f_1,\ldots,f_d,c) \in \mathcal{F}_{a,M_1,M_2}^{\bL_d}\Big\}.
\]
We first claim that each $(\check{\mathcal{F}}_{a,M_1,M_2}^{\bL_d})_j$ is a $P_{\bX}$-Glivenko--Cantelli class, where $P_{\bX}$ is the distribution of $\bX$.  To see this, note that by Theorem~2.7.5 of \citet{vanderVaartWellner1996}, there exists a universal constant $C > 0$ and functions $g_k^L,g_k^U:\mathbb{R} \rightarrow [0,1]$ for $k=1,\ldots,N_1$ with $N_1 = e^{2M_1 C/\epsilon}$ such that $\mathbb{E}|g_k^U(X_{1j}) - g_k^L(X_{1j})| \leq \epsilon/(2M_1)$ and such that for every monotone function $g:\mathbb{R} \rightarrow [0,1]$, we can find $k^* \in \{1,\ldots,N_1\}$ with $g_{k^*}^L \leq g \leq g_{k^*}^U$.  By Corollary~2.7.10, the same property holds for convex or concave functions from $[-a,a]$ to $[0,1]$, provided we use $N_2$ brackets, where $N_2 = \exp\bigl\{C\bigl(1+ \frac{M_2}{2M_1}\bigr)^{1/2}(2M_1/\epsilon)^{1/2}\bigr\}$.  It follows that if $j$ corresponds to a monotone component, then the class of functions
\[
\tilde{g}_k^L(x) = 2M_1(g_k^L(x_j)-1/2)\mathbbm{1}_{\{\bx \in [-a,a]^d\}}, \quad \tilde{g}_k^U(x) = 2M_1(g_k^U(x_j)-1/2)\mathbbm{1}_{\{\bx \in [-a,a]^d\}},
\]
for $k=1,\ldots,N_1$, forms an $\epsilon$-bracketing set for $\big(\check{\mathcal{F}}_{a,M_1,M_2}^{\bL_d}\big)_j$ in the $L_1(P_{\bX})$-norm.  Similarly, if $j$ corresponds to a convex or concave component, we can define in the same way an $\epsilon$-bracketing set for $\big(\check{\mathcal{F}}_{a,M_1,M_2}^{\bL_d}\big)_j$ of cardinality $N_2$ for $\big(\check{\mathcal{F}}_{a,M_1,M_2}^{\bL_d}\big)_j$.  We deduce by Theorem~2.4.1 of \citet{vanderVaartWellner1996} that each $\big(\check{\mathcal{F}}_{a,M_1,M_2}^{\bL_d}\big)_j$ is a $P_X$-Glivenko--Cantelli class.  But then
\[
\sup_{\check{f} \in \check{\mathcal{F}}_{a,M_1,M_2}^{\bL_d}} \bigg|\frac{1}{n}\sum_{i=1}^n \check{f}(\bX_i) - \mathbb{E}\check{f}(\bX)\bigg| \le \sum_{j=1}^d\sup_{\check{f}_j \in \big(\check{\mathcal{F}}_{a,M_1,M_2}^{\bL_d}\big)_j} \bigg|\frac{1}{n} \sum_{i=1}^n \check{f}_j(X_{ij}) - \mathbb{E} \check{f}_j(X_{1j})\bigg|,
\]
so $\check{\mathcal{F}}_{a,M_1,M_2}^{\bL_d}$ is $P_{\bX}$-Glivenko--Cantelli.  We now use this fact to show that the class of functions
\[
\mathcal{H}_{a,M_1,M_2} = \Big\{ h_f: \mathbb{R}^d \times \mathbb{R} \rightarrow \mathbb{R} \;\big|\; h_f(\bx,y) = \big\{yf(\bx) - B(f(\bx))\big\}\mathbbm{1}_{\{\bx \in [-a,a]^d\}}, f \in \mathcal{F}_{a,M_1,M_2}^{\bL_d} \Big\}
\]
is $P$-Glivenko--Cantelli, where $P$ is the distribution of $(\bX,Y)$.  Define $f^*,f^{**}: \mathbb{R}^d \times \mathbb{R} \rightarrow \mathbb{R}$ by $f^*(\bx,y) =y$ and $f^{**}(\bx,y) = \mathbbm{1}_{\{\bx \in [-a,a]^d\}}$.  Let 
\[
\mathcal{F}_1 = \bigl\{f:\mathbb{R}^d \times \mathbb{R} \rightarrow \mathbb{R}\bigm| f(\bx,y) = \check{f}(\bx), \check{f} \in\check{\mathcal{F}}_{a,M_1,M_2}^{\bL_d}\bigr\},
\]
let $\mathcal{F}_2 = \{f^*\}$ and let $\mathcal{F}_3 = \{f^{**}\}$; finally define $\psi:\mathbb{R} \times \mathbb{R} \times \mathbb{R} \rightarrow \mathbb{R}$ by $\psi(u,v,w) = \{vu - B(u)\}w$.  Then $\mathcal{H} = \psi(\mathcal{F}_1,\mathcal{F}_2,\mathcal{F}_3)$, where 
\[
\psi(\mathcal{F}_1,\mathcal{F}_2,\mathcal{F}_3) = \{\psi(f_1(\bx,y),f_2(\bx,y),f_3(\bx,y)):f_1 \in \mathcal{F}_1,f_2 \in \mathcal{F}_2,f_3 \in \mathcal{F}_3\}.
\]
Now $\mathcal{F}_1, \mathcal{F}_2$ and $\mathcal{F}_3$ are $P$-Glivenko--Cantelli, $\psi$ is continuous and (recalling that $|Y| \leq 1$ in the Binomial setting), 
\[
\sup_{f_1 \in \mathcal{F}_1} \sup_{f_2 \in \mathcal{F}_2} \sup_{f_3 \in \mathcal{F}_3} |\psi(f_1(\bx,y),f_2(\bx,y),f_3(\bx,y))| \leq M_1(d+1) + B(M_1(d+1)),
\]
which is $P$-integrable.  We deduce from Theorem~3 of \citet{vanderVaartWellner2000} that $\mathcal{H}_{a,M_1,M_2}$ is $P$-Glivenko--Cantelli.

\textbf{Step 5: Almost sure convergence of $\hat{f_n}$}. For $\epsilon > 0$, let
\[
	B_{\epsilon}(f_0) = \bigg\{f:\mathbb{R}^d\rightarrow\mathbb{R} \bigg| \sup_{\bx \in [-a_0,a_0]^d} |f(\bx)-f_0(\bx)| \le \epsilon \bigg\},
\]
where we suppress the dependence of $B_{\epsilon}(f_0)$ on $a_0$ in the notation.
Our aim to show that with probability 1, we have $\hat{S}_n^{\bL_d} \cap \mathrm{cl}\big(\mathcal{F}^{\bL_d} \setminus B_{\epsilon}(f_0) \big) = \emptyset$ for sufficiently large~$n$.  In Lemma~\ref{Lem:positive} in the online supplementary material, it is established that for any $\epsilon > 0$, 
\begin{align}
\notag	\zeta(a^*): = \mathbb{E} \bigl[\big\{ Yf_0(\bX) - &B(f_0(\bX))\big\}\mathbbm{1}_{\{\bX \in [-a^*,a^*]^d\}}\bigr] \\
\label{Eq:consistency0} & - \sup_{f \in \mathcal{F}_{a_0,M(a_0),W(a_0)}^{\bL_d} \backslash B_{\epsilon}(f_0)} \mathbb{E}\bigl[\big\{ Yf(\bX) - B(f(\bX)) \big\} \mathbbm{1}_{\{\bX \in [-a^*,a^*]^d\}}\bigr]
\end{align}
is positive and a non-decreasing function of $a^* > a_0+1$.  Since we also have that (in the Binomial setting), $-\log 2 \leq g^{-1}(t)t - B(t) \leq 0$, we can therefore choose $a^* > a_0+1$ such that 
\begin{align}
\label{Eq:consistency1} \bigl|\mathbb{E} \bigl[\big\{ Y f_0(\bX) - B(f_0(\bX))\big\} \mathbbm{1}_{\{ \bX \notin [-a^*,a^*]^d\}}\bigr] \bigr| & \le \zeta(a^*)/3.
\end{align}
Let 
\begin{align*}
\mathcal{F}^* = \mathrm{cl}\big(\mathcal{F}_{a_0,M(a_0)}^{\bL_d}\backslash B_{\epsilon}(f_0)\big)\; &\cap \;  \mathrm{cl}\big(\mathcal{F}_{a_0+1,M(a_0+1)}^{\bL_d} \backslash B_{\epsilon}(f_0)\big)  \\
&\cap \;  \mathrm{cl}\big(\mathcal{F}_{a^*,M(a^*)}^{\bL_d} \backslash B_{\epsilon}(f_0)\big) \; \cap \;  \mathrm{cl}\big(\mathcal{F}_{a^*+1,M(a^*+1)}^{\bL_d} \backslash B_{\epsilon}(f_0)\big).
\end{align*}
Observe that by the result of Step~2, we have that with probability one, $\hat{S}_n^{\bL_d} \subseteq \mathcal{F}^* \cup \mathrm{cl}(B_\epsilon(f_0))$ for sufficiently large $n$.  By Lemma~\ref{Lem:Intersect2} in the online supplementary material, 
\begin{align}
\notag&\bigg\{ \sup_{f \in \mathcal{F}^*} \frac{1}{n}\sum_{i=1}^n\big\{Y_i f(\bX_i) - B(f(\bX_i))\big\} \ge \bar{L}_0  -  \zeta(a^*)/3 \bigg\} \\
\label{Eq:consistency2}	&\subseteq  \Bigg\{ \sup_{f \in (\mathcal{F}_{a_0,M(a_0),W(a_0)}^{\bL_d} \cap \mathcal{F}_{a^*,M(a^*)+1,W(a^*)+1}^{\bL_d}) \backslash B_{\epsilon}(f_0)} \frac{1}{n}\sum_{i=1}^n\big\{Y_i f(\bX_i) - B(f(\bX_i))\big\} \mathbbm{1}_{\{\bX_i \in [-a^*,a^*]^d\}} \\
\notag	& \qquad \qquad \qquad + \sup_{f \in \mathcal{F}^{\bL_d}}\frac{1}{n}\sum_{i=1}^n\big\{Y_i f(\bX_i) - B(f(\bX_i))\big\} \mathbbm{1}_{\{\bX_i \notin [-a^*,a^*]^d\}} \ge \bar{L}_0 - \zeta(a^*)/3 \Bigg\}, 
\end{align}
Here the closure operator in~(\ref{Eq:consistency2}) can be dropped by the same argument as in Step~2.  Now note that 
\begin{align*}
	\Big\{ h_f: \mathbb{R}^d \times \mathbb{R} \rightarrow \mathbb{R} \;\big|\;  &h_f(\bx,y) = \big\{yf(\bx) - B(f(\bx))\big\}\mathbbm{1}_{\{\bx \in [-a^*,a^*]^d\}}, \\
&f \in (\mathcal{F}_{a_0,M(a_0),W(a_0)}^{\bL_d} \cap \mathcal{F}_{a^*,M(a^*)+1,W(a^*)+1}^{\bL_d}) \backslash B_{\epsilon}(f_0) \Big\} \subseteq \mathcal{H}_{a^*,M(a^*)+1,W(a^*)+1},
\end{align*}
so the class is $P$-Glivenko--Cantelli, by the result of Step~4.  We therefore have that with probability one, 
\begin{align}
\notag	\limsup_{ n\rightarrow \infty} &\sup_{f \in (\mathcal{F}_{a_0,M(a_0),W(a_0)}^{\bL_d} \cap \mathcal{F}_{a^*,M(a^*)+1,W(a^*)+1}^{\bL_d}) \backslash B_{\epsilon}(f_0)} \frac{1}{n}\sum_{i=1}^n\big\{Y_i f(\bX_i) - B(f(\bX_i))\big\} \mathbbm{1}_{\{\bX_i \in [-a^*,a^*]^d\}}\\
\notag  & = \sup_{f \in (\mathcal{F}_{a_0,M(a_0),W(a_0)}^{\bL_d} \cap \mathcal{F}_{a^*,M(a^*)+1,W(a^*)+1}^{\bL_d}) \backslash B_{\epsilon}(f_0)} \mathbb{E}\bigl[\big\{Y f(\bX) - B(f(\bX))\big\} \mathbbm{1}_{\{\bX \in [-a^*,a^*]^d\}}\bigr] \\
\label{Eq:consistency3}		&\leq \mathbb{E} \bigl[\big\{ Yf_0(\bX) - B(f_0(\bX))  \big\} \mathbbm{1}_{\{\bX \in [-a^*,a^*]^d\}}\bigr] - \zeta(a^*) \\
\label{Eq:consistency4}		&\le \bar{L}_0 - 2\zeta(a^*)/3, 
\end{align}
where (\ref{Eq:consistency3}) is due to (\ref{Eq:consistency0}), and where (\ref{Eq:consistency4}) is due to (\ref{Eq:consistency1}).  In addition, under the Binomial setting, for every $n \in \mathbb{N}$,
\begin{align}
\label{Eq:consistency5}
\sup_{f \in \mathcal{F}^{\bL_d}} \frac{1}{n}\sum_{i=1}^n\big\{Y_i f(\bX_i) - B(f(\bX_i))\big\} \mathbbm{1}_{\{\bX_i \notin [-a^*,a^*]^d\}} \le \frac{1}{n}\sum_{i=1}^n\sup_{t \in \mathbb{R}} \big\{ Y_i t - B(t)\big\} \mathbbm{1}_{\{\bX_i \notin [-a^*,a^*]^d\}} \le 0.
\end{align}
We deduce from~(\ref{Eq:consistency2}), (\ref{Eq:consistency4}) and (\ref{Eq:consistency5}) that with probability one, $\hat{S}_n^{\bL_d} \subseteq \mathrm{cl}\big(B_{\epsilon}(f_0) \big)$ for sufficiently large $n$. Finally, since $\mathrm{cl}\big(B_{\epsilon}(f_0) \big)|_{[-a_0,a_0]^d} = B_{\epsilon}(f_0)|_{[-a_0,a_0]^d}$, the conclusion of Theorem~\ref{Thm:SCMLEconsistency} for Binomial models follows.

\textbf{Consistency of other EF additive models}.  The proof for other EF models follows the same structure, but involves some changes in certain places.  We list the modifications required for each step here:
\begin{itemize}
\setlength{\itemsep}{0pt}
\setlength{\parskip}{0pt}
\setlength{\parsep}{0pt}
\item In Step~1, we add a term independent of $f$ to the definition of the partial log-likelihood: 
\[
\tilde{\ell}_n(f) = \frac{1}{n}\sum_{i=1}^n \biggl[Y_if(\bX_i) - B(f(\bX_i)) - \sup_{t \in \mathrm{dom}(B)} \{Y_i t - B(t)\}\biggr]. 
\]
Note that
\[
\sup_{t \in \mathrm{dom}(B)} \{Y_i t - B(t)\} = \left\{ \begin{array}{ll} Y_i^2/2 & \mbox{if EF is Gaussian;} \\
Y_i \log Y_i - Y_i & \mbox{if EF is Poisson;} \\
-1 - \log Y_i & \mbox{if EF is Gamma.} \end{array} \right.
\]
This allows us to prove that $\mathbb{E}\{\tilde{\ell}_n(f_0)\} \in (-\infty,0]$ in all cases: in particular, in the Gaussian case, $\mathbb{E}\{\tilde{\ell}_n(f_0)\} = -\phi_0/4$; for the Poisson, we can use Lemma~\ref{Lem:Pois} in the online supplementary material to see that $\mathbb{E}\{\tilde{\ell}_n(f_0)\} \in [-1,0]$; for the Gamma, this claim follows from Lemma~\ref{Lem:Gamma} in the online supplementary material. It then follows from the strong law of large numbers that almost surely
\[
\liminf_{n \rightarrow \infty} \sup_{f \in \mathrm{cl}(\mathcal{F}^{\bL_d})} \tilde{\ell}_n(f) \ge \mathbb{E}\{\tilde{\ell}_n(f_0)\} =: \tilde{L}_0.
\]

\item In Step~2,  the deterministic constant $M = M(a) \in (0, \infty)$ needs to be chosen differently for different EF distributions. Let $\mathcal{C}_a = \{C_1,\ldots,C_N\}$ be the same finite collection of compact subsets defined previously. We then can pick  
\begin{align*}
M = \left\{ \begin{array}{ll} 
4\Bigl(\sqrt{\frac{2(-\tilde{L}_0+1)}{\min_{1 \le k \le N} \mathbb{P}(\bX \in C_k, |Y| \le 1)}} + 1\Bigl) & \mbox{if EF is Gaussian;} \\
4\Bigl(\frac{-\tilde{L}_0+1}{\min_{1 \le k \le N} \mathbb{P}(\bX \in C_k, Y = 1)} + 1\Bigr) & \mbox{if EF is Poisson;} \\
4\Bigl(\frac{2(-\tilde{L}_0+1)}{\min_{1 \le k \le N} \mathbb{P}(\bX \in C_k, 1 \le Y \le e)} + 4 \Bigr) & \mbox{if EF is Gamma.} \end{array} \right.
\end{align*}

\item Step~3 are exactly the same for all the EF distributions listed in Table~\ref{Tab:glmformula}. 

\item In Step~4, we define the class of functions
\begin{align*}
\tilde{\mathcal{H}}_{a,M_1,M_2} &= \Big\{ h_f: \mathbb{R}^d \times \mathbb{R} \rightarrow \mathbb{R} \;\big|\;  \\ 
&h_f(\bx,y) = \Big[yf(\bx) - B(f(\bx)) -  \sup_{t \in \mathrm{dom}(B)} \{y t - B(t)\}\Big] \mathbbm{1}_{\{\bx \in [-a,a]^d\}}, f \in \mathcal{F}_{a,M_1,M_2}^{\bL_d} \Big\}.
\end{align*}
In the Gaussian case, we can rewrite $h_f(\bx,y) = -\frac{1}{2}\{y-f(\bx)\}^2\mathbbm{1}_{\{\bx \in [-a,a]^d\}}$. By taking the $P$-integrable envelope function to be 
\[
F(\bx,y) = \frac{1}{2}\{|y| + M_1(d+1)\}^2 \mathbbm{1}_{\{\bx \in  [-a,a]^d\}} \ge \sup_{f \in \mathcal{F}_{a,M_1,M_2}^{\bL_d}}|h_f(\bx,y)|,
\]
we can again deduce from Theorem~3 of \citet{vanderVaartWellner2000} that $\tilde{\mathcal{H}}_{a,M_1,M_2}$ is $P$-Glivenko--Cantelli. Similarly, in the Poisson case, we can show that $\tilde{\mathcal{H}}_{a,M_1,M_2}$ is $P$-Glivenko--Cantelli by taking the envelope function to be $F(\bx,y) = \bigl\{ y M_1(d+1) + e^{M_1(d+1)} + y + y \log y \bigr\} \mathbbm{1}_{\{\bx \in  [-a,a]^d\}}$.

The Gamma case is slightly more complex, mainly due to the fact that $\mathrm{dom}(B) \neq \mathbb{R}$. For $\delta > 0$, let
\begin{align*}
\tilde{\mathcal{H}}_{a,M_1,M_2}^{\delta} &= \Big\{h_f: \mathbb{R}^d \times \mathbb{R} \rightarrow \mathbb{R} \;\big|\;  \\ &h_f(\bx,y) = \big\{yf(\bx) + \log\bigl(\max(-f(\bx),\delta)\bigr) -  1 + \log y \big\} \mathbbm{1}_{\{\bx \in [-a,a]^d\}}, f \in \mathcal{F}_{a,M_1,M_2}^{\bL_d} \Big\}.
\end{align*}
Again, we can show that $\tilde{\mathcal{H}}_{a,M_1,M_2}^\delta$ is $P$-Glivenko--Cantelli by taking the envelope function for $\tilde{\mathcal{H}}_{a,M_1,M_2}^{\delta}$ to be $F(\bx,y) = \bigl\{ y M_1(d+1) + |\log \delta| + |\log(M_1(d+1))| + 1 + \log y \bigr\} \mathbbm{1}_{\{\bx \in  [-a,a]^d\}}$.

\item Step~5 for the Gaussian and Poisson settings are essentially a replication of that for the Binomial case. Only very minor changes are required:
\begin{enumerate}[(a)]
\setlength{\itemsep}{0pt}
\setlength{\parskip}{0pt}
\setlength{\parsep}{0pt}
\item where applicable, add the term $- \sup_{t \in \mathbb{R}} [Y t - B(t)]$  to $\{Yf_0(\bX) - B(f_0(\bX))\}$ and $\{Yf(\bX) - B(f(\bX))\}$; make the respective change to $\{Y_i f(\bX_i) - B(f(\bX_i))\}$ and $\{yf(\bx) - B(f(\bx))\}$;
\item change $\bar{L}_0$ to $\tilde{L}_0$;
\item change $\mathcal{H}_{a^*,M(a^*)+1,W(a^*)+1}$ to $\tilde{\mathcal{H}}_{a^*,M(a^*)+1,W(a^*)+1}$;
\item rewrite  (\ref{Eq:consistency5}) as 
\[
\sup_{f \in \mathcal{F}^{\bL_d}} \frac{1}{n}\sum_{i=1}^n\Big[Y_i f(\bX_i) - B(f(\bX_i)) - \sup_{t \in \mathbb{R}} \{Y_i t - B(t)\}\Big] \mathbbm{1}_{\{\bX_i \notin [-a^*,a^*]^d\}} \le 0.
\] 
\end{enumerate}
The analogue of Step~5 for the Gamma distribution is a little more involved.  Set $\delta_0 = \inf_{\bx \in [-a_0-1,a_0+1]^d} -f_0(\bx) / e^2 > 0$. Note that the above supremum is attained as $f_0$ is a continuous function. Then one can prove in a similar fashion to Lemma~\ref{Lem:positive} that
\begin{align*}
\zeta &=  \mathbb{E} \bigl[\big\{ Yf_0(\bX) + \log(-f_0(\bX)) + 1 + \log Y\big\}\mathbbm{1}_{\{\bX \in [-a_0-1,a_0+1]^d\}}\bigr] \\
& - \sup_{f \in \mathcal{F}_{a_0,M(a_0),W(a_0)}^{\bL_d} \backslash B_{\epsilon}(f_0)} \! \! \! \! \! \! \! \! \mathbb{E}\bigl[\big\{ Yf(\bX) + \log(\max(-f(\bX),\delta_0)) + 1 + \log Y \big\} \mathbbm{1}_{\{\bX \in [-a_0-1,a_0+1]^d\}}\bigr] > 0.
\end{align*}
Next we pick $a^* > a_0+1$ such that 
\begin{align*}
\bigl|\mathbb{E} \bigl[\big\{ Yf_0(\bX) + \log(-f_0(\bX)) + 1 + \log Y\big\} \mathbbm{1}_{\{ \bX \notin [-a^*,a^*]^d\}}\bigr] \bigr| & \le \zeta/3
\end{align*}
and $\delta^* = \inf_{\bx \in [-a^*-1,a^*+1]^d} -f_0(\bx) / e^2$.  Write
\[
\mathcal{F}^{**} = \Bigl(\mathcal{F}_{a_0,M(a_0),W(a_0)}^{\bL_d} \cap \mathcal{F}_{a^*,M(a^*)+1,W(a^*)+1}^{\bL_d} \Bigr) \backslash B_{\epsilon}(f_0).
\] 
With $\mathcal{F}^*$ defined as in Step 5, we have
\begin{align*}
&\bigg\{ \sup_{f \in \mathcal{F}^* } \frac{1}{n}\sum_{i=1}^n\big\{Y_i f(\bX_i) + \log(-f(\bX_i)) + 1 + \log Y_i \big\} \ge \tilde{L}_0  -  \zeta/3 \bigg\}\\
&\subseteq  \Bigg\{ \sup_{f \in \mathcal{F}^{**}} \frac{1}{n}\sum_{i=1}^n\big\{Y_i f(\bX_i) + \log(\max(-f(\bX_i),\delta_0)) + 1 + \log Y_i \big\} \mathbbm{1}_{\{\bX_i \in [-a_0-1,a_0+1]^d\}} \\
& \qquad + \sup_{f \in \mathcal{F}^{**}} \frac{1}{n}\sum_{i=1}^n\big\{Y_i f(\bX_i) + \log(\max(-f(\bX_i),\delta^*)) + 1 + \log Y_i \big\} \mathbbm{1}_{\{\bX_i \in [-a^*,a^*]^d \backslash [-a_0-1,a_0+1]^d\}}\\
& \qquad + \sup_{f \in \mathcal{F}^{\bL_d}}\frac{1}{n}\sum_{i=1}^n\big\{Y_i f(\bX_i) + \log(-f(\bX_i)) + 1 + \log Y_i\big\} \mathbbm{1}_{\{\bX_i \notin [-a^*,a^*]^d\}} \ge \tilde{L}_0 - \zeta/3 \Bigg\}. 
\end{align*}
Again we apply Glivenko--Cantelli theorem to finish the proof, where we also use the fact that
\begin{align*}
\sup_{f \in \mathcal{F}^{**}} \mathbb{E}\bigl[\big\{ Yf(\bX) &+ \log(\max(-f(\bX),\delta^*)) + 1 + \log Y \big\} \mathbbm{1}_{\{\bX \in [-a^*,a^*]^d \backslash [-a_0-1,a_0+1]^d\}} \bigr] \\
&\leq \mathbb{E} \bigl[\big\{ Yf_0(\bX) + \log(-f_0(\bX)) + 1 + \log Y\big\} \mathbbm{1}_{\{\bX \in [-a^*,a^*]^d \backslash [-a_0-1,a_0+1]^d\}}\bigr]. 
 \end{align*}
\end{itemize}
\end{prooftitle}
\bigskip

\begin{prooftitle}{of Corollary~\ref{Cor:SCMLEconsistency}}
By Theorem~\ref{Thm:SCMLEconsistency}, we have
\[
\sup_{\hat{f}_n \in \hat{S}_n^{\bL_d}}|\hat{c}_n - c_0| = \sup_{\hat{f}_n \in \hat{S}_n^{\bL_d}} |\hat{f}_n(\mathbf{0}) - f_0(\mathbf{0})| \stackrel{a.s.}{\rightarrow} 0
\]
as $n \rightarrow \infty$.  Moreover, writing $I_j = \{0\} \times \ldots \times \{0\} \times [-a_0,a_0] \times \{0\} \times \ldots \times \{0\}$, we have
\[
\sup_{\hat{f}_n \in \hat{S}_n^{\bL_d}} \sum_{j=1}^d \sup_{x_j \in [-a_0, a_0]}|\hat{f}_{n,j}(x_j) - f_{0,j}(x_j)| = \sup_{\hat{f}_n \in \hat{S}_n^{\bL_d}} \sum_{j=1}^d \sup_{\bx \in I_j} |\hat{f}_n(\bx) - f_0(\bx) - \hat{c}_n + c_0| \stackrel{a.s.}{\rightarrow} 0,
\]
using Theorem~\ref{Thm:SCMLEconsistency} again and the triangle inequality.
\end{prooftitle}
\bigskip

\begin{prooftitle}{of Proposition~\ref{Prop:SCAIEcontinuity}}
Fix an index matrix $\bA = (\balpha_1,\ldots,\balpha_m) \in \mathbb{R}^{d \times m}$. For any sequence $\bA^1, \bA^2, \ldots \in \mathbb{R}^{d \times m}$ with $\lim_{k \rightarrow \infty} \|\bA^k-\bA\|_F = 0$, where $\|\cdot\|_F$ denotes the Frobenius norm, we claim that 
$\lim_{k \rightarrow \infty} \| (\bA^k)^T \bX_i - \bA^T \bX_i \|_1 = 0$ for every $i = 1,\ldots,n$. To see this, we write $\bA^k = (\balpha_1^k, \ldots, \balpha_m^k)$. It then follows that
\begin{align*}
\| (\bA^k)^T \bX_i - \bA^T \bX_i \|_1 &= \sum_{h=1}^m \Big|((\bA^k)^T \bX_i)_h - (\bA^T \bX_i)_h \Big| = \sum_{h=1}^m \Big|\sum_{j=1}^d (A_{jh}^k - A_{jh}) X_{ij} \Big| \\
&\le \sum_{h=1}^m \bigg[ \|\bX_i\|_2 \bigg\{\sum_{j=1}^d (A_{jh}^k - A_{jh})^2 \bigg\}^{1/2} \bigg] \le  \|\bX_i\|_2 \sqrt{m} \|\bA^k-\bA\|_F \rightarrow 0
\end{align*} 
as $k \rightarrow \infty$, where we have applied the Cauchy--Schwarz inequality twice.  Now write $\bZ_i = (Z_{i1},\ldots,Z_{im})^T = \bA^T \bX_i$ for every $i=1,\ldots,n$ and take 
\[
	a^* = \max_{1 \le i \le n, 1 \le j \le m} |Z_{ij}|.
\]
Since $\bigcup_{M=1}^\infty \mathcal{F}^{\bL_m}_{a^*,M} = \mathcal{F}^{\bL_m}$ (where $\mathcal{F}^{\bL_m}_{a^*,M}$ is defined in (\ref{Eq:setFLaM})), it follows that 
\[
	\lim_{M \rightarrow \infty} \sup_{f \in \mathcal{F}^{\bL_m}_{a^*,M}} \bar{\ell}_{n,m}\big(f; (\bZ_1,Y_1),\ldots,(\bZ_n,Y_n)\big) = \sup_{f \in \mathcal{F}^{\bL_m}} \bar{\ell}_{n,m}\big(f; (\bZ_1,Y_1),\ldots,(\bZ_n,Y_n)\big) = \Lambda_n(\bA).
\] 
Therefore, for any $\epsilon > 0$, there exist $M_{\epsilon} > 0$ and $f^* \stackrel{\mathcal{F}^{\bL_m}}{\sim} (f_1^*,\ldots,f_m^*,c^*) \in  \mathcal{F}^{\bL_m}_{a^*,M_{\epsilon}}$ such that
\[
	\bar{\ell}_{n,m}\Big( f^*;(\bZ_1,Y_1),\ldots,(\bZ_n,Y_n)\Big) \ge \Lambda_n(\bA) - \epsilon.
\]
We can then find piecewise linear and continuous functions $f_1^{**},\ldots,f_m^{**}$ such that $f_j^{**}(Z_{ij}) = f_j^{*}(Z_{ij})$ for every $i=1,\ldots,n$, $j = 1,\ldots,m$. Consequently, the additive function $f^{**}(\bz) = \sum_{j=1}^m f_j^{**}(z_j) + c^*$ is continuous.  It now follows that
\begin{align*}
	\liminf_{k \rightarrow \infty} \Lambda_n(\bA^k) &\ge \liminf_{k \rightarrow \infty} \bar{\ell}_{n,m}\big(f^{**};((\bA^k)^T \bX_1,Y_1),\ldots,((\bA^k)^T \bX_n,Y_n)\big) \\
	&= \bar{\ell}_{n,m}\big(f^{**};(\bZ_1,Y_1),\ldots,(\bZ_n,Y_n)\big) \\
&= \bar{\ell}_{n,m}\big(f^{*};(\bZ_1,Y_1),\ldots,(\bZ_n,Y_n)\big) \ge \Lambda_n(\bA) - \epsilon.
\end{align*}
Since both $\epsilon > 0$ and the sequence $(\bA^k)$ were arbitrary, the result follows.
\end{prooftitle}
\bigskip

\begin{prooftitle}{of Theorem~\ref{Thm:SCAIEconsistency}}
The structure of the proof is essentially the same as that of Theorem~\ref{Thm:SCMLEconsistency}.  For the sake of brevity, we focus on the main changes and on the Gaussian setting.  Following the strategy used in the proof of Theorem~\ref{Thm:SCMLEconsistency}, we work here with the logarithm of a normalised likelihood:
\begin{align*}
\tilde{\ell}_{n,m}(f; \bA) &\equiv \tilde{\ell}_{n,m}\bigl(f; (\bA^T \bX_1,Y_1),\ldots,(\bA^T \bX_n,Y_n) \bigr)\\
 &= \bar{\ell}_{n,m}\bigl(f; (\bA^T \bX_1,Y_1),\ldots,(\bA^T \bX_n,Y_n) \bigr) - \frac{1}{n}\sum_{i=1}^n \sup_{t \in \mathrm{dom}(B)} \{Y_i t - B(t)\} \\ 
&= -\frac{1}{2n}\sum_{i=1}^n \bigl\{f(\bA^T \bX_i) - Y_i\bigr\}^2. 
\end{align*}
So in Step~1, we can establish that $\mathbb{E}\tilde{\ell}_{n,m}(f_0; \bA_0) = -\phi_0/4$.

In Step~2, we aim to bound $\tilde{f}_n^I$ on $[-a,a]^d$ for any fixed $a > 0$. Three cases are considered:
\begin{enumerate}[(a)]
\setlength{\itemsep}{0pt}
\setlength{\parskip}{0pt}
\setlength{\parsep}{0pt}
\item If $m \ge 2$ and $\bL_m \in \mathcal{L}_m$, then $\tilde{f}_n^I$ is either convex or concave. One can now use the convexity/concavity to show that $\limsup_{n \rightarrow \infty} \sup_{\bx \in [-a,a]^d} |\tilde{f}_n^I(\bx)| < M(a)$ almost surely for some deterministic constant $M(a) < \infty$ that only depends on $a$. See, for instance, Proposition~4 of \citet{LimGlynn2012} for a similar argument.
\item Otherwise, if $\bL_m \notin \mathcal{L}_m$, we will show that there exists deterministic $M(a) \in (0,\infty)$ such that with probability one, 
\begin{align}
\label{Eq:AddIndexFuncBound1}
	\tilde{S}^{\bL_m}_n \subseteq \mathcal{G}_{a,M(a)}^{\bL_m, \delta}
\end{align}
for sufficiently large $n$, where we define 
\begin{align*}
	\mathcal{G}^{\bL_m,\delta}_{a,M} = \Big\{ f^I:\mathbb{R}^d\rightarrow \mathbb{R} \; \Big| \; f^I(\bx) = f(\bA^T \bx), \mbox{ with } f \in \mathcal{F}^{\bL_m}_{a,M} \mbox{ and } \bA \in \mathcal{A}^{\bL_m,\delta}_d \Big\}.
\end{align*}
To see this, we first extend Lemma~\ref{Lem:AddFuncBound} to the following:
\begin{lem}
\label{Lem:AddIndexFuncBound}
Fix $a > 0$ and $\delta > 0$, and set $\tilde{\delta}= \min(\delta,d^{-1})$.  For every $f^I(\bx) = f(\bA^T \bx) = \sum_{j=1}^m f_j(\balpha_j^T \bx) + c$ with $f \stackrel{\mathcal{F}^{\bL_m}}{\sim} (f_1,\ldots,f_m,c)$ and $\bA \in \mathcal{A}^{\bL_m,\delta}_d$, there exists a convex, compact subset $D_{f^I}$ of $[-2\tilde{\delta}^{-1/2}ad, 2\tilde{\delta}^{-1/2}ad]^d$ having Lebesgue measure $\bigl(\frac{a}{2d}\bigr)^d$ such that
\begin{align}
\label{Eq:IndexFuncBound2}
	\max\Bigl\{ \inf_{\bx \in D_{f^I}} f^I(\bx), \inf_{\bx \in D_{f^I}} -f^I(\bx)\Bigr\} \ge  \frac{1}{4} \max \biggl\{\sup_{|z_1|\le a} |f_1(z_1)|, \ldots, \sup_{|z_m|\le a} |f_m(z_m)|, 2 |c|\biggr \}.
\end{align}
\end{lem}
\begin{proof}
First consider the case $m = d$. Note that every $\bA \in \mathcal{A}^{\bL_d,\delta}_d$ is invertible.  In fact, if $\lambda$ is an eigenvalue of $\bA$, then $\delta^{1/2} \leq |\lambda| \leq 1$, where the upper bound follows from the Gerschgorin circle theorem \citep{Gerschgorin1931,GradshteynRyzhik2007}.  Let $C_1, \ldots, C_N$ be the sets constructed for $f$ in Lemma~\ref{Lem:AddFuncBound}.  Then, writing $\nu_d$ for Lebesgue measure on $\mathbb{R}^d$,
\[
\min_{1 \le k \le N} \nu_d\big((\bA^{T})^{-1}C_k\big) \ge  \frac{1}{|\mathrm{det}(\bA^T)|} \min_{1 \le k \le N}  \nu_d(C_k) \geq \bigg(\frac{a}{2d}\bigg)^d,
\]
and
\[
\bigcup_{1 \le k \le N}(\bA^{T})^{-1} C_k \subseteq (\bA^{T})^{-1} [-2a,2a]^d \subseteq [-2\tilde{\delta}^{-1/2}ad, 2\tilde{\delta}^{-1/2}ad]^d.
\]
Thus~(\ref{Eq:IndexFuncBound2}) is satisfied.  To complete the proof of this lemma, we note that for any $m < d$, we can always find a $d \times (d-m)$ matrix $\bB=(\bbeta_1,\ldots,\bbeta_{d-m})$ such that 
\begin{enumerate}[1.]
\setlength{\itemsep}{0pt}
\setlength{\parskip}{0pt}
\setlength{\parsep}{0pt}
\item $\|\bbeta_j\|_1 = 1$ for every $j = 1,\ldots,d-m$.
\item $\bbeta_j^T \bbeta_k = 0$ for every $1 \le j < k \le d-m$. 
\item $\bA^T \bB = 0$.
\end{enumerate}
Let $\bA_+ = (\bA,\bB)$, so the modulus of every eigenvalue of $\bA_+$ belongs to $[\min(\delta^{1/2},d^{-1/2}),1]$.  Since $f^I(\bx)= f(\bA^T \bx) \equiv f'({\bA}_{+}^{T} \bx)$ with $f'(\bz) = \sum_{j=1}^{m}f_j(z_j)+c$ for every $\bz = (z_1,\ldots,z_d)^T\in \mathbb{R}^d$, the problem reduces to the case $m=d$. 
\end{proof}
\bigskip

Then, instead of using the strong law of large numbers to complete this step, we apply the Glivenko--Cantelli theorem for classes of convex sets \citep[Theorem 1.11]{BhattacharyaRao1976}.  This change is necessary to circumvent the fact that the set $D_{f^I}$ depends on the function $f^I$ (via its index matrix $\bA$). 

\item Finally, if $m = 1$, then the Cauchy--Schwarz inequality gives that $\mathcal{A}_d^{L_1} \equiv \mathcal{A}^{L_1,\delta}_d$ with $\delta = d^{-1}$. Thus (\ref{Eq:AddIndexFuncBound1}) still holds true. 
\end{enumerate}

Two different cases are considered in Step~4:
\begin{enumerate}[(a)]
\setlength{\itemsep}{0pt}
\setlength{\parskip}{0pt}
\setlength{\parsep}{0pt}
\item If $m \ge 2$ and $\bL_m \in \mathcal{L}_m$, then without loss of generality, we can assume $\bL_m \in \{1,4,5,6\}^m$. It is enough to show that the set of functions
\begin{align*}
	\mathcal{G}^{\bL_m,d}_{a,M_1,M_2} = &\Big\{ h_f:\mathbb{R}^d \times \mathbb{R} \rightarrow \mathbb{R} \; \Big| \; h_f(\bx,y) = -\frac{1}{2}\bigl\{ f(\bx) - y \bigr\}^2 \mathbbm{1}_{\{\bx \in [-a,a]^d\}} \mbox{ with } f: \mathbb{R}^d \rightarrow \mathbb{R} \mbox{ convex}, 
\\&\sup_{\bx \in [-a,a]^d} |f(\bx)| \le M_1 \mbox{ and } |f(\bx_1)-f(\bx_2)| \le M_2 \|\bx_1 -\bx_2\| \mbox{ for any } \bx_1, \bx_2 \in [-a,a]^d \Big\}
\end{align*}
is $P$-Glivenko--Cantelli, where $P$ is the distribution of $(\bX,Y)$. This follows from an application of Corollary~2.7.10 and Theorem~2.4.1 of \citet{vanderVaartWellner1996}, as well as Theorem~3 of \citet{vanderVaartWellner2000}.

\item Otherwise, we need to show that the set of functions 
\begin{align*}
	\mathcal{G}^{\bL_m,d,\delta}_{a,M_1,M_2} = \Big\{ h_{f,\bA}:\mathbb{R}^d \times \mathbb{R} \rightarrow \mathbb{R} \; \Big| \; h_{f,\bA}(\bx,y) = &-\frac{1}{2}\bigl\{ f(\bA^T \bx) - y \bigr\}^2 \mathbbm{1}_{\{\bx \in [-a,a]^d\}} \\ &\mbox{ with } f \in \mathcal{F}^{\bL_m}_{a,M_1,M_2} \mbox{ and } \bA \in \mathcal{A}^{\bL_m,\delta}_d \Big\}
\end{align*}
is $P$-Glivenko--Cantelli. The proof is similar to that given in Step~4 of the proof of Theorem~\ref{Thm:SCMLEconsistency}.  The compactness of $\mathcal{A}^{\bL_m,\delta}_d$, together with a bracketing number argument is used here to establish the claim. See Lemma~\ref{Lem:GAIMGC} in the online supplementary material for details.
\end{enumerate}
\end{prooftitle}
\bigskip

\begin{prooftitle}{of Corollary~\ref{Cor:SCAIEconsistency}}
This result follows from Theorem~1 of \citet{Yuan2011} and our Theorem~\ref{Thm:SCAIEconsistency}. See also Theorem~5 of \citet{SamworthYuan2012} for a similar type of argument. 
\end{prooftitle}


\newpage
\section*{ONLINE SUPPLEMENTARY MATERIAL}

Recall the definition of $\Theta$ from the proof of Proposition~\ref{Prop:gamexistunique}.
\begin{lem}
\label{Lem:Closed}
The set $\Theta$ is a closed subset of $\bar{\mathbb{R}}^n$.
\end{lem}
\begin{proof}
Suppose that, for each $m \in \mathbb{N}$, the vector $\boldsymbol{\eta}^m = (\eta_1^m,\ldots,\eta_n^m)^T$ belongs to $\Theta$, and that $\boldsymbol{\eta}^m \rightarrow \boldsymbol{\eta} = (\eta_1,\ldots,\eta_n)^T$ as $m \rightarrow \infty$.  Then, for each $m \in \mathbb{N}$, there exists a sequence $(f^{m,k}) \in \mathcal{F}^{\bL_d}$ such that $f^{m,k}(\bX_i) \rightarrow \eta_i^m$ as $k \rightarrow \infty$ for $i=1,\ldots,n$.  It follows that we can find $k_m \in \mathbb{N}$ such that $f^{m,k_m}(\bX_i) \rightarrow \eta_i$ as $m \rightarrow \infty$, for each $i=1,\ldots,n$.

For $j=1,\ldots,d$, let $\{X_{(i),j}\}_{i=1}^{N_j}$ denote the distinct order statistics of $\{X_{ij}\}_{i=1}^n$ (thus $N_j < n$ if there are ties among $\{X_{ij}\}_{i=1}^n$).  Moreover, let
\[
\mathcal{V}_j = \{(-\infty,X_{(1),j}],[X_{(1),j},X_{(2),j}],\ldots,[X_{(N_j-1),j},X_{(N_j),j}],[X_{(N_j),j},\infty)\},
\]
and let $\mathcal{V} = \times_{j=1}^d \mathcal{V}_j$.  Thus $|\mathcal{V}| = \prod_{j=1}^d (N_j+1)$ and the union of all the sets in $\mathcal{V}$ is $\mathbb{R}^d$.  Writing $f^{m,k_m} \stackrel{\mathcal{F}^{\bL_d}}{\sim} (f_1^{m,k_m},\ldots,f_d^{m,k_m},c^{m,k_m})$, we define a modified sequence $\tilde{f}^m \stackrel{\mathcal{F}^{\bL_d}}{\sim} (\tilde{f}_1^m,\ldots,\tilde{f}_d^m,\tilde{c}^m)$ at $\bx = (x_1,\ldots,x_d)^T \in \mathbb{R}^d$ by setting 
\begin{align*}
\tilde{f}_j^m(x_j) = \left\{ \begin{array}{ll} \frac{(X_{(i+1),j} - x_j)f_j^{m,k_m}(X_{(i),j})}{X_{(i+1),j} - X_{(i),j}} + \frac{(x_j - X_{(i),j})f_j^{m,k_m}(X_{(i+1),j})}{X_{(i+1),j} - X_{(i),j}} & \mbox{if $x_j \in [X_{(i),j},X_{(i+1),j}]$} \\
\frac{(X_{(2),j} - x_j)f_j^{m,k_m}(X_{(1),j})}{X_{(2),j} - X_{(1),j}} + \frac{(x_j - X_{(1),j})f_j^{m,k_m}(X_{(2),j})}{X_{(2),j} - X_{(1),j}} & \mbox{if $x_j \in (-\infty,X_{(1),j}]$} \\
\frac{(X_{(N_j),j} - x_j)f_j^{m,k_m}(X_{(N_j-1),j})}{X_{(N_j),j} - X_{(N_j-1),j}} + \frac{(x_j - X_{(N_j-1),j})f_j^{m,k_m}(X_{(N_j),j})}{X_{(N_j),j} - X_{(N_j-1),j}} & \mbox{if $x_j \in [X_{(N_j),j},\infty)$,} \end{array} \right.
\end{align*}
and $\tilde{c}^m = c^{m,k_m}$.  Thus each component function $\tilde{f}_j^m$ is piecewise linear, continuous and satisfies the same shape constraint as $f_j^{m,k_m}$, and $\tilde{f}^m$ is piecewise affine and $\tilde{f}^m(\bX_i) = f^{m,k_m}(\bX_i) = \eta_i^m$ for $i=1,\ldots,n$.  The proof will therefore be concluded if we can show that a subsequence of $(\tilde{f}^m)$ converges pointwise in $\bar{\mathbb{R}}$.  To do this, it suffices to show that, given an arbitrary $V \in \mathcal{V}$, we can find a subsequence of $(\tilde{f}^m|_V)$ (where $\tilde{f}^m|_V$ denotes the restriction of $\tilde{f}^m$ to $V$) converging pointwise in $\bar{\mathbb{R}}$.  Note that we can write
\[
\tilde{f}^m|_V(\bx) = (\ba^m)^T(\bx^T,1)^T 
\]
for some $\ba^m = (a_1^m,\ldots,a_{d+1}^m)^T \in \mathbb{R}^{d+1}$.  If the sequence $(\ba^m)$ is bounded, then we can find a subsequence $(\ba^{m_k})$, converging to $\ba \in \mathbb{R}^{d+1}$, say.  In that case, for all $\bx \in V$, we have $\tilde{f}^{m_k}|_V(\bx) \rightarrow \ba^T(\bx^T,1)^T$, and we are done.  On the other hand, if $(\ba^m)$ is unbounded, we can let $j^m = \argmax_{j = 1,\ldots,d+1} |a_j^m|$, where we choose the largest index in the case of ties.  Since $j^m$ can only take $d+1$ values, we may assume without loss of generality that there is a subsequence $(j^{m_k})$ such that $j^{m_k} = d+1$ for all $k \in \mathbb{N}$ and such that $a_{d+1}^{m_k} \rightarrow \infty$ as $k \rightarrow \infty$.  By choosing further subsequences if necessary, we may also assume that
\[
\biggl(\frac{a_1^{m_k}}{a_{d+1}^{m_k}},\ldots,\frac{a_d^{m_k}}{a_{d+1}^{m_k}}\biggr)^T \rightarrow (\tilde{a}_1,\ldots,\tilde{a}_d)^T =: \tilde{\ba},
\]
say, where $\tilde{\ba} \in [-1,1]^d$.  Writing $V_1 = \{\bx \in V: (\tilde{\ba}^T,1)(\bx^T,1)^T = 0\}$, $V_1^+ = \{\bx \in V: (\tilde{\ba}^T,1)(\bx^T,1)^T > 0\}$ and $V_1^- = \{\bx \in V: (\tilde{\ba}^T,1)(\bx^T,1)^T < 0\}$, we deduce that for large $k$,
\[
\tilde{f}^{m_k}|_V(\bx) = a_{d+1}^{m_k}\biggl(\frac{a_1^{m_k}}{a_{d+1}^{m_k}},\ldots,\frac{a_d^{m_k}}{a_{d+1}^{m_k}},1\biggr)^T(\bx^T,1)^T \rightarrow \left\{ \begin{array}{ll} \infty & \mbox{if $\bx \in V_1^+$} \\
-\infty & \mbox{if $\bx \in V_1^-$.} \end{array} \right.
\]
It therefore suffices to consider $\tilde{f}^{m_k}|_{V_1}$.  We may assume that $\tilde{\ba} \neq \mathbf{0}$ (otherwise $V_1 = \emptyset$ and we are done), so without loss of generality assume $\tilde{a}_d \neq 0$.  But then, for $\bx \in V_1$,
\[
\tilde{f}^{m_k}|_{V_1}(\bx) = (\ba^{m_k})^T(\bx^T,1)^T = (\bb^{m_k})^T(\bx_{(-d)}^T,1)^T,
\]
where $\bx_{(-d)} = (x_1,\ldots,x_{d-1})^T$, and where $\bb^{m_k} = (b_1^{m_k},\ldots,b_d^{m_k}) \in \mathbb{R}^d$, with $b_j^{m_k} = a_j^{m_k} - \frac{a_d^{m_k}}{\tilde{a}_d}\tilde{a}_j$ for $j=1,\ldots,d-1$ and $b_d^{m_k} = a_{d+1}^{m_k} - \frac{a_d^{m_k}}{\tilde{a}_d}$.  Applying the same argument inductively, we find subsets $V_2,\ldots,V_{d+1}$, where $V_1 \supseteq V_2 \supseteq \ldots \supseteq V_{d+1}$, where $V_j$ has dimension $d-j$ and $V_{d+1} = \emptyset$, such that a subsequence of $(\tilde{f}^{m_k})$ converges pointwise in $\bar{\mathbb{R}}$ for all $\bx \in V \setminus V_j$.
\end{proof}
\bigskip

Now recall the definitions of $\mathcal{F}_{a,M}^{\bL_d}$, $\mathcal{F}_{a,M_1,M_2}^{\bL_d}$, $M(a)$ and $W(a)$ from the proof of Theorem~\ref{Thm:SCMLEconsistency}.
\begin{lem}
\label{Lem:Intersect}
For any $a>0$, we have $\mathrm{cl}\big(\mathcal{F}_{a,M(a)}^{\bL_d}\big) \cap \mathrm{cl}\big(\mathcal{F}_{a+1,M(a+1)}^{\bL_d}\big)  \subseteq \mathrm{cl}\big(\mathcal{F}_{a,M(a),W(a)}^{\bL_d}\big)$.
\end{lem}
\begin{proof}
We first consider the case $M(a) \le M(a+1)$. Suppose $f \in \mathrm{cl}\big(\mathcal{F}_{a,M(a)}^{\bL_d}\big) \cap \mathrm{cl}\big(\mathcal{F}_{a+1,M(a+1)}^{\bL_d}\big)$, so there exists a sequence $(f^k)$ such that $f^k \in \mathcal{F}_{a,M(a)}^{\bL_d}$ and such that $f^k \stackrel{\mathcal{F}^{\bL_d}}{\sim} (f_1^k,\ldots,f_d^k,c^k)$ converges pointwise in $\bar{\mathbb{R}}$ to $f$.  Our first claim is that there exists a subsequence $(f^{k_m})$ such that $f^{k_m} \in \mathcal{F}_{a+1,M(a+1)+1}^{\bL_d}$ for every $m \in \mathbb{N}$. 

Indeed, suppose for a contradiction that there exists $K \in \mathbb{N}$ such that for every $k \geq K$, we have $f^k \notin \mathcal{F}_{a+1,M(a+1)+1}^{\bL_d}$. Let
\[
\bb^k = (b^k_1,\ldots,b^k_{2d+1})^T = \Bigl(|f_1^k(-a-1)|, |f_1^k(a+1)|, \ldots, |f_d^k(-a-1)|, |f_d^k(a+1)|, 2|c^k|\Bigr )^T.
\]
It follows from our hypothesis and the shape restrictions that $\max_{j=1,\ldots,2d+1} b_j^k > M(a+1)+1$ for $k \geq K$. Furthermore, we cannot have $\argmax_{j=1,\ldots,2d+1} b_j^k = 2d+1$ for any $k \geq K$, because $2|c^k| = 2|f^k(\mathbf{0})| \leq M(a) < M(a+1) + 1$ for every $k \in \mathbb{N}$.  We therefore let $j^k = \argmax_{j = 1,\ldots,2d} b_j^k$, where we choose the largest index in the case of ties.  Since $j^k$ can only take $2d$ values, we may assume without loss of generality that there is a subsequence $(j^{k_m})$ such that $j^{k_m} = 2d$ for all $m \in \mathbb{N}$.  But, writing $\bx_0 = (0,\ldots,0,a+1)^T \in \mathbb{R}^d$, this implies that 
\[
|f(\bx_0) - f(\mathbf{0})| = \lim_{m \rightarrow \infty} |f^{k_m}(\bx_0)-f^{k_m}(\mathbf{0})| = \lim_{m \rightarrow \infty} |f_d^{k_m}(a+1)| \geq M(a+1)+1.
\]
On the other hand, since $f \in \mathrm{cl}\big(\mathcal{F}_{a+1,M(a+1)}^{\bL_d}\big)$, we can find $(\tilde{f}^m) \in \mathcal{F}_{a+1,M(a+1)}^{\bL_d}$ such that $\tilde{f}^m \stackrel{\mathcal{F}^{\bL_d}}{\sim} (\tilde{f}_1^m,\ldots,\tilde{f}_d^m,\tilde{c}^m)$ converges pointwise in $\bar{\mathbb{R}}$ to $f$.  So
\[
|f(\bx_0) - f(\mathbf{0})| = \lim_{m \rightarrow \infty} |\tilde{f}^m(\bx_0)-\tilde{f}^m(\mathbf{0})| = \lim_{m \rightarrow \infty} |\tilde{f}_d^m(a+1)| \leq M(a+1).
\]
This contradiction establishes our first claim.  Since $\mathcal{F}_{a,M(a)}^{\bL_d} \cap \mathcal{F}_{a+1,M(a+1)+1}^{\bL_d} \subseteq \mathcal{F}_{a,M(a),W(a)}^{\bL_d}$, we deduce that $f \in \mathrm{cl}\big(\mathcal{F}_{a,M(a),W(a)}^{\bL_d}\big)$ in the case where $M(a) \le M(a+1)$.

Now if $M(a) > M(a+1)$, then for every $f \in \mathrm{cl}\big(\mathcal{F}_{a,M(a)}^{\bL_d}\big) \cap \mathrm{cl}\big(\mathcal{F}_{a+1,M(a+1)}^{\bL_d}\big)$,  there exists a sequence $(f^k)$ such that $f^k \in \mathcal{F}_{a+1,M(a+1)}^{\bL_d}$ and such that $f^k$ converges pointwise in $\bar{\mathbb{R}}$ to $f$. By the shape restrictions, $\mathcal{F}_{a+1,M(a+1)}^{\bL_d} \subseteq \mathcal{F}_{a,M(a)}^{\bL_d}$, so $f^k \in \mathcal{F}_{a,M(a)}^{\bL_d}$. Consequently, $f^k \in \mathcal{F}_{a,M(a),W(a)}^{\bL_d}$ as above, so $f \in \mathrm{cl}\big(\mathcal{F}_{a,M(a),W(a)}^{\bL_d}\big)$.
\end{proof}
\bigskip

\begin{lem}
\label{Lem:positive}
Under assumptions \textbf{(A.1)} - \textbf{(A.4)}, for any $a,M_1, M_2, \epsilon > 0$,
\begin{align*}
\mathbb{E} \big[\bigl\{ Yf_0(\bX) - B(f_0(\bX))  \big\}&\mathbbm{1}_{\{\bX \in [-a-1,a+1]^d\}}\bigr] \\
&> \sup_{f \in \mathcal{F}_{a,M_1,M_2}^{\bL_d} \backslash B_{\epsilon}(f_0)} \mathbb{E}\big[\bigl\{ Yf(\bX) - B(f(\bX)) \big\} \mathbbm{1}_{\{\bX \in [-a-1,a+1]^d\}}\bigr].
\end{align*}
\end{lem}
\begin{proof}
Since $B' = g^{-1}$, we have that for every $\bx \in [-a-1,a+1]^d$, the expression 
\[
\mathbb{E} \big\{ Yf(\bX) - B(f(\bX)) | \bX = \bx \big\} = g^{-1}(f_0(\bx))f(\bx) - B(f(\bx))
\]
is uniquely maximised by taking $f(\bx) = f_0(\bx)$. Moreover, since $f_0$ is continuous by assumption \textbf{(A.4)}, it is uniformly continuous on $[-a-1,a+1]^d$.  We may therefore assume that for any $\epsilon' > 0$, there exists $\gamma(\epsilon') > 0$ such that $|f_{0,j}(z_1)-f_{0,j}(z_2)|<\epsilon'$ for every $j = 1,\ldots,d$ and every $z_1,z_2 \in [-a-1,a+1]$ with $|z_1-z_2| < \gamma(\epsilon')$.  For any $f \in \mathcal{F}_{a,M_1,M_2}^{\bL_d} \backslash B_{\epsilon}(f_0)$, there exists $\bx^* = (x_1^*,\ldots,x_d^*)^T \in [-a,a]^d$ such that $|f(\bx^*)-f_0(\bx^*)| > \epsilon$.  Let $C_{\bx^*,1} = \times_{j=1}^d D_j \subseteq [-a-1,a+1]^d$ where 
\[
D_j = 
\begin{cases}
[x^*_j,x^*_j + \min\{\gamma(\frac{\epsilon}{2d}),1\}] & \mbox{ if } l_j = 2 \\
[x^*_j - \min\{\gamma(\frac{\epsilon}{2d}),1\},x^*_j] & \mbox{ if } l_j = 3 \\ 
\Bigl[x^*_j - \min \Big\{\frac{1}{M_2}\frac{\epsilon}{4d}, \gamma(\frac{\epsilon}{4d}), 1\Big\},x^*_j+\min \Big\{\frac{1}{M_2}\frac{\epsilon}{4d}, \gamma(\frac{\epsilon}{4d}), 1\Big\}\Bigr] & \mbox{ if } l_j \in \{1,4,5,6,7,8,9\}.
\end{cases}
\]
Define $C_{\bx^*,2}$ similarly, but with the intervals in the cases $l_j=2$ and $l_j=3$ exchanged.  Then the shape constraints ensure that $\max\{\inf_{\bx \in C_{\bx^*,1}} |f(\bx)-f_0(\bx)|,\inf_{\bx \in C_{\bx^*,2}} |f(\bx)-f_0(\bx)|\} > \epsilon/2$.  But the $d$-dimensional Lebesgue measures of $C_{\bx^*,1}$ and $C_{\bx^*,2}$ do not depend on $\bx^*$, and $\min\{\mathbb{P}(\bX \in C_{\bx^*,1}),\mathbb{P}(\bX \in C_{\bx^*,2})\}$ is a continuous function of $\bx^*$, so by (\textbf{A.2}), we have 
\[
\xi = \inf_{\bx^* \in [-a,a]^d} \min\{\mathbb{P}(\bX \in C_{\bx^*,1}),\mathbb{P}(\bX \in C_{\bx^*,2})\} > 0.
\]  
Moreover, writing $\underline{f}_0 = \inf_{\bx \in [-a-1,a+1]^d} f_0(\bx)$ and $\overline{f}_0 = \sup_{\bx \in [-a-1,a+1]^d} f_0(\bx)$, and using the fact that $s \mapsto \bigl[\{g^{-1}(f_0(\bx))f_0(\bx) - B(f_0(\bx))\} - \{ g^{-1}(f_0(\bx))s - B(\bx) \}\bigr]$ is convex, we deduce that 
\begin{align*}
\mathbb{E} \big[ &\bigl\{Yf_0(\bX) - B(f_0(\bX))\bigr\} \mathbbm{1}_{\{\bX \in [-a-1,a+1]^d\}}\bigr] \\
&- \sup_{f \in \mathcal{F}_{a,M_1,M_2}^{\bL_d} \setminus B_{\epsilon}(f_0)} \mathbb{E}\big[ \bigl\{Yf(\bX) - B(f(\bX)) \big\} \mathbbm{1}_{\{\bX \in [-a-1,a+1]^d\}}\bigr] \\
&\geq \xi \inf_{\bx \in [-a-1,a+1]^d} \inf_{|t-f_0(\bx)| > \epsilon/2} \bigl[\{g^{-1}(f_0(\bx))f_0(\bx) - B(f_0(\bx))\} - \{g^{-1}(f_0(\bx))t - B(t)\}\bigr] \\
&\geq \frac{1}{16}\xi \epsilon^2 \inf_{s \in [\underline{f}_0-\epsilon/2,\overline{f}_0+\epsilon/2]} (g^{-1})'(s) > 0.
\end{align*}
\end{proof}
\bigskip
\begin{lem}
\label{Lem:Intersect2}
For any $a^* > a_0+1$, we have 
\begin{align*}
&\mathrm{cl}\Big(\mathcal{F}_{a_0,M(a_0)}^{\bL_d}\backslash B_{\epsilon}(f_0)\Big)\; \cap \;  \mathrm{cl}\Big(\mathcal{F}_{a_0+1,M(a_0+1)}^{\bL_d} \backslash B_{\epsilon}(f_0)\Big)  \; \cap \;  \mathrm{cl}\Big(\mathcal{F}_{a^*,M(a^*)}^{\bL_d} \backslash B_{\epsilon}(f_0)\Big) \; \\
&\cap \;  \mathrm{cl}\Big(\mathcal{F}_{a^*+1,M(a^*+1)}^{\bL_d} \backslash B_{\epsilon}(f_0)\Big) \subseteq \mathrm{cl}\bigg(\big(\mathcal{F}_{a_0,M(a_0),W(a_0)}^{\bL_d} \cap \mathcal{F}_{a^*,M(a^*)+1,W(a^*)+1}^{\bL_d}\big) \backslash B_{\epsilon}(f_0) \bigg).
\end{align*}
\end{lem}
\begin{proof}
The proof is very similar indeed to the proof of Lemma~\ref{Lem:Intersect}, so we omit the details.
\end{proof}
\bigskip

Recall the definition of $\tilde{l}_n(f_0)$ from the proof of Theorem~\ref{Thm:SCMLEconsistency}.
\begin{lem}
\label{Lem:Pois}
Suppose that $Z$ has a Poisson distribution with mean $\mu \in (0,\infty)$. Then 
\begin{align*}
\mu \log \mu \le \mathbb{E}(Z \log Z) \le \mu \log \mu + 1.
\end{align*}
It follows that, under the Poisson setting, $\mathbb{E}\{\tilde{l}_n(f_0)\} \in [-1,0]$.
\end{lem}
\begin{proof}
The lower bound is immediate from Jensen's inequality.  For the upper bound, let $Z_0 = (Z - \mu)/\sqrt{\mu}$, so $\mathbb{E}(Z_0) = 0$ and $\mathbb{E}(Z_0^2) = 1$. It follows from the inequality $\log (1+z) \le z$ for any $z > -1$ that
\begin{align*}
\mathbb{E}(Z \log Z) & = \mathbb{E}\bigl[(\mu + \sqrt{\mu}Z_0)\{\log \mu + \log(1+ Z_0/\sqrt{\mu})\}\mathbbm{1}_{\{Z_0 > -\sqrt{\mu}\}}\bigr] \\
& \le \mathbb{E}\bigl[(\mu + \sqrt{\mu}Z_0)(\log \mu +  Z_0/\sqrt{\mu})\bigr]\\
& = \mu \log \mu + (\log \mu +1 ) \sqrt{\mu} \mathbb{E}(Z_0) + \mathbb{E}(Z_0^2) =  \mu \log \mu  + 1.
\end{align*}
Finally, we note that
\begin{align*}
\mathbb{E}\tilde{l}_n(f_0) &= \mathbb{E}\bigl[\mathbb{E}\bigl\{Yf_0(\bX) - B(f_0(\bX)) - Y \log Y + Y | \bX\bigr\} \bigr] \\
&= \mathbb{E}\bigl\{e^{f_0(\bX)}f_0(\bX) - \mathbb{E}(Y \log Y | \bX)\bigr\} \in [-1,0].
\end{align*}
\end{proof}
\bigskip

\begin{lem}
\label{Lem:Gamma}
In the Gamma setting, under assumption \textbf{(A.1)} and \textbf{(A.3)}, $\mathbb{E}\{\tilde{l}_n(f_0)\} \in (-\infty,0)$.
\end{lem}
\begin{proof}
Since $-Yf_0(\bX) | \bX \sim \Gamma(1/\phi_0,1/\phi_0)$, we have
\begin{align*}
\mathbb{E}\tilde{l}_n(f_0) = \mathbb{E}\bigl[\mathbb{E}\bigl\{Yf_0(\bX) - B(f_0(\bX)) - \log Y + 1 | \bX \bigr\} \bigr] &= \mathbb{E}\bigl[\mathbb{E}\bigl\{\log (-Yf_0(\bX)) | \bX\bigr\}\bigr] \\
&= \log \phi_0 + \psi_D(1/\phi_0) \in (-\infty,0),
\end{align*}
where $\psi_D(\cdot)$ denotes the digamma function.
\end{proof}
\bigskip

\begin{lem}
\label{Lem:GAIMGC}
In the Gaussian setting, under \textbf{(A.1)}-\textbf{(A.2)} and  \textbf{(B.2)}-\textbf{(B.3)},
\begin{align*}
	\mathcal{G}^{\bL_m,d,\delta}_{a,M_1,M_2} = \Big\{ h_{f,\bA}:\mathbb{R}^d \times \mathbb{R} \rightarrow \mathbb{R} \; \Big| \; h_{f,\bA}(\bx,y) = &-\frac{1}{2}\bigl\{ f(\bA^T \bx) - y \bigr\}^2 \mathbbm{1}_{\{\bx \in [-a,a]^d\}} \\ &\mbox{ with } f \in \mathcal{F}^{\bL_m}_{a,M_1,M_2} \mbox{ and } \bA \in \mathcal{A}^{\bL_m}_{d,\delta} \Big\}
\end{align*}
is $P$-Glivenko--Cantelli.
\end{lem}
\begin{proof}
Following the argument in Step~4 of the proof of Theorem~\ref{Thm:SCMLEconsistency}, it suffices to show that
\begin{align*}
\big(\mathring{\mathcal{F}}_{a,M_1,M_2}^{\bL_d}\big)_j = \Big\{\mathring{f} : \mathbb{R}^d \rightarrow \mathbb{R}\Big| \mathring{f}(\bx) = f_j(\balpha_j^T \bx)\mathbbm{1}_{\{\bx \in [-a,a]^d\}} \text{ for some } f &\stackrel{\mathcal{F}^{\bL_d}}{\sim} (f_1,\ldots,f_m,c) \in \mathcal{F}_{a,M_1,M_2}^{\bL_m} \\ &\mbox{and } \balpha_j \in \mathbb{R}^d \mbox{ with } \|\balpha_j\|_1 = 1 \Big\}
\end{align*}
is $P$-Glivenko--Cantelli for every $j = 1,\ldots,m$. In the following, we present the proof in case $l_j = 2$. Other cases can be shown in a similar manner. 

By Theorem~2.7.5 of \citet{vanderVaartWellner1996}, there exists a universal constant $C > 0$ such that for any $\epsilon > 0$ and any $\balpha_0 \in \mathbb{R}^d$, there exist functions $g_k^L,g_k^U:\mathbb{R} \rightarrow [0,1]$ for $k=1,\ldots,N_3$ with $N_3 = e^{4M_1 C/\epsilon}$ such that $\mathbb{E}|g_k^U(\balpha_0^T \bX) - g_k^L(\balpha_0^T \bX)| \leq \epsilon/(4M_1)$ and such that for every monotone function $g:\mathbb{R} \rightarrow [0,1]$, we can find $k^* \in \{1,\ldots,N_3\}$ with $g_{k^*}^L \leq g \leq g_{k^*}^U$. Since $\bX$ has a Lebesgue density, for every $k$ we can find $\tau_k^L, \tau_k^U > 0$ such that 
\[
\mathbb{E}\bigl| g_k^L(\balpha_0^T \bX) - g_k^L(\balpha_0^T \bX - \tau_k^L)\bigr| \le \frac{\epsilon}{8M_1} \quad \mbox{ and } \quad \mathbb{E}\bigl| g_k^U(\balpha_0^T \bX + \tau_k^U) - g_k^L(\balpha_0^T \bX)\bigr| \le \frac{\epsilon}{8M_1}.
\]
By picking $\tau = \min\{\tau_1^L,\ldots,\tau_N^L, \tau_1^U,\ldots,\tau_N^U\}/a$ (which implicitly depends on $\balpha_0$), we claim that the class of functions 
\[
\tilde{g}_k^L(\bx) = 2M_1(g_k^L(\balpha_0^T \bx - \tau a)-1/2)\mathbbm{1}_{\{\bx \in [-a,a]^d\}}, \quad \tilde{g}_k^U(\bx) = 2M_1(g_k^U(\balpha_0^T \bx + \tau a)-1/2)\mathbbm{1}_{\{\bx \in [-a,a]^d\}}
\]
for $k=1,\ldots,N_3$, form an $\epsilon$-bracketing set in the $L_1(P_{\bX})$-norm for the set of functions
\begin{align*}
\notag \mathring{\mathcal{F}}_{a,M_1}^{\balpha_0,\tau} = \Big\{\mathring{f} : \mathbb{R}^d \rightarrow \mathbb{R}\Big| \mathring{f}(\bx) = f(\balpha^T \bx)&\mathbbm{1}_{\{\bx \in [-a,a]^d\}}, \text{ with } f: \mathbb{R} \rightarrow \mathbb{R} \mbox{ increasing, }\\ & \sup_{x \in \mathbb{R}} |f(x)| \le M_1 \mbox{ and } \|\balpha - \balpha_0\|_1 \le \tau \Big\}.
\end{align*}
To see this, we note that
\[
\sup_{\|\balpha-\balpha_0\|_1 \le \tau, \bx \in [-a,a]^d} |\balpha^T \bx - \balpha_0^T \bx | \le \tau a.
\]
It follows by monotonicity that for $k = 1, \ldots, N_3$, 
\begin{align*}
\mathbb{E} |\tilde{g}_k^U(\bX) - \tilde{g}_k^L(\bX)| &\le 2M_1 \mathbb{E} \bigl|g_k^U(\balpha_0^T \bX + \tau a) - g_k^U(\balpha_0^T \bX)\bigr|  
 + 2M_1 \mathbb{E}\bigl|g_k^U(\balpha_0^T \bX) - g_k^L(\balpha_0^T \bX)\bigr| \\ 
&\hspace{1cm} + 2M_1 \mathbb{E}\bigl| g_k^L(\balpha_0^T \bX) - g_k^L(\balpha_0^T \bX - \tau a) \bigr| \leq \frac{\epsilon}{4} + \frac{\epsilon}{2} + \frac{\epsilon}{4} = \epsilon. 
\end{align*}
Therefore, $\{\tilde{g}_k^L, \tilde{g}_k^U\}_{k=1}^{N_3}$ is indeed an $\epsilon$-bracketing set.

Now for every $\balpha_0 \in \mathbb{R}^d$ with $\|\balpha_0\|_1 = 1$, we can pick $\tau(\balpha_0) > 0$ such that a finite $\epsilon$-bracketing set can be found for $\mathring{\mathcal{F}}_{a,M_1}^{\balpha_0,\tau(\balpha_0)}$.  Since $\{\balpha_0 \in \mathbb{R}^d: \|\balpha_0\|_1 = 1\}$ is compact, we can pick $\balpha_0^1, \ldots, \balpha_0^{N^*}$ such that 
\[
\{\balpha_0 \in \mathbb{R}^d: \|\balpha_0\|_1 = 1\} \subseteq \bigcup_{k=1,\ldots,N^*} \bigl\{\balpha \in \mathbb{R}^d: \|\balpha - \balpha_0^k\|_1 \leq  \tau(\balpha_0^k) \bigr\}.
\]
Consequently, for every $\epsilon > 0$, a finite $\epsilon$-bracketing set can be found for $\big(\mathring{\mathcal{F}}_{a,M_1,M_2}^{\bL_d}\big)_j$. Finally, we complete the proof by applying Theorem~2.4.1 of \citet{vanderVaartWellner1996}.
\end{proof}
\newpage

\begin{table}[!htbp]\small
  \centering
  \begin{tabular}{|c|c|c|c|c|c|}
  \multicolumn{6}{c}{\normalsize Average running time per dataset (in seconds): Gaussian} \\ 
  \multicolumn{6}{c}{Problem 1} \\ \hline
  Method & $n=200$ & $n=500$ & $n=1000$ & $n=2000$ & $n=5000$ \\ \hline
  SCMLE   & 0.13 &  0.34 & 0.90 & 1.86 & 7.35  \\
  SCAM    & 0.91 &  1.72 & 4.17 & 7.43 &18.59  \\
  GAMIS   & 0.11 &  0.20 & 0.46 & 1.46 & 3.93  \\
  MARS    & 0.01 &  0.01 & 0.02 & 0.05 & 0.12  \\
  Tree    & 0.01 &  0.01 & 0.02 & 0.03 & 0.08  \\
  CAP     & 0.61 &  1.75 & 2.47 & 3.86 & 8.60  \\
  MCR     &30.17 &411.80 & - & - & -  \\\hline
  \multicolumn{6}{c}{Problem 2} \\ \hline
  Method & $n=200$ & $n=500$ & $n=1000$ & $n=2000$ & $n=5000$ \\ \hline
  SCMLE   & 0.07 &  0.18 & 0.37 & 0.88 & 3.03  \\
  SCAM    & 2.85 &  3.27 & 6.26 &12.22 &29.78  \\
  GAMIS   & 0.11 &  0.20 & 0.44 & 1.39 & 3.92  \\
  MARS    & 0.01 &  0.01 & 0.02 & 0.04 & 0.09  \\
  Tree    & 0.01 &  0.01 & 0.02 & 0.03 & 0.07  \\
  CAP     & 0.11 &  0.32 & 0.55 & 0.97 & 1.93  \\
  MCR     &33.31 &427.98 & - & - & -  \\\hline
  \multicolumn{6}{c}{Problem 3} \\ \hline
  Method & $n=200$ & $n=500$ & $n=1000$ & $n=2000$ & $n=5000$ \\ \hline
  SCMLE   & 0.35 &  0.95 & 2.37 & 5.41 & 20.21  \\
  SCAM    &23.08 & 25.77 &38.60 &70.67 &143.91  \\
  GAMIS   & 0.45 &  0.60 & 1.10 & 3.19 &  8.09  \\
  MARS    & 0.01 &  0.02 & 0.04 & 0.08 &  0.22  \\
  Tree    & 0.01 &  0.02 & 0.03 & 0.05 &  0.12  \\
  CAP     & 0.10 &  0.37 & 0.99 & 1.83 &  4.20  \\
  MCR     &26.61 &303.40 & - & - & -  \\\hline
  \end{tabular}
  \caption{\small Average running time (in seconds) of SCMLE, SCAM, GAMIS, MARS, Tree, CAP and MCR on problems 1, 2, 3 with sample sizes $n = 200, 500, 1000, 2000, 5000$ in the Gaussian setting.}
  \label{Tab:simtimeG}
\end{table}

\begin{table}[!p]\small
  \centering
  \begin{tabular}{|c|c|c|c|c|c|c|}
  \multicolumn{7}{c}{\normalsize Average running time per dataset (in seconds): Poisson and Binomial} \\ 
  \multicolumn{7}{c}{Problem 1} \\ \hline
  Model & Method & $n=200$ & $n=500$ & $n=1000$ & $n=2000$ & $n=5000$ \\ \hline
           & SCMLE   & 0.33 & 0.78 & 1.76 & 3.98 & 13.08  \\
  Poisson  & SCAM    & 1.24 & 2.40 & 4.92 & 9.99 & 30.54  \\
           & GAMIS   & 0.25 & 0.50 & 1.00 & 2.43 &  7.08  \\\hline
           & SCMLE   & 0.24 & 0.53 & 1.23 & 3.22 &  9.51  \\
  Binomial & SCAM    & 0.80 & 1.09 & 1.92 & 5.24 &  9.06  \\
           & GAMIS   & 0.25 & 0.47 & 0.93 & 2.49 &  6.66  \\\hline
  \multicolumn{7}{c}{Problem 2} \\ \hline
  Model & Method & $n=200$ & $n=500$ & $n=1000$ & $n=2000$ & $n=5000$ \\ \hline
           & SCMLE   & 0.20 & 0.41 & 0.84 &  1.80 &  5.10  \\
  Poisson  & SCAM    & 1.97 & 2.67 & 5.17 & 11.34 & 25.35  \\
           & GAMIS   & 0.24 & 0.42 & 0.94 &  2.43 &  6.62  \\\hline
           & SCMLE   & 0.16 & 0.35 & 0.72 & 1.49  &  4.63  \\
  Binomial & SCAM    & 1.82 & 3.06 & 6.38 & 9.60  & 25.87  \\
           & GAMIS   & 0.24 & 0.47 & 0.94 & 2.34  &  6.59  \\\hline
  \multicolumn{7}{c}{Problem 3} \\ \hline
  Model & Method & $n=200$ & $n=500$ & $n=1000$ & $n=2000$ & $n=5000$ \\ \hline
           & SCMLE   & 0.90 &  2.29 &  5.59 & 12.68 & 42.58  \\
  Poisson  & SCAM    & 8.85 & 16.93 & 22.77 & 39.69 & 77.08  \\
           & GAMIS   & 0.91 &  1.62 &  2.99 &  7.02 & 19.01  \\\hline
           & SCMLE   & 0.46 & 1.10 & 2.50 &  5.37 & 18.54  \\
  Binomial & SCAM    & 5.80 & 6.29 & 8.73 & 14.10 & 30.07  \\
           & GAMIS   & 1.18 & 1.53 & 2.83 &  6.93 & 16.41  \\\hline
    \end{tabular}
  \caption{\small Average running time (in seconds) of SCMLE, SCAM and GAMIS on problems 1, 2, 3 with sample sizes $n = 200, 500, 1000, 2000, 5000$ in the Poisson and Binomial settings.}
  \label{Tab:simtimePB}
\end{table}
\begin{table}[!p]\small
  \centering
  \begin{tabular}{|c|c|c|c|c|c|}
  \multicolumn{6}{c}{\normalsize Average running time per dataset (in seconds):} \\ 
  \multicolumn{6}{c}{\normalsize Additive Index Models} \\ 
  \multicolumn{6}{c}{Problem 4} \\ \hline
  Method & $n=200$ & $n=500$ & $n=1000$ & $n=2000$ & $n=5000$ \\ \hline
  SCAIE   & 4.36  &  6.61 &  12.20 & 23.50  & 69.52\\
  SSI     & 26.10 &112.44 & 411.16 &1855.37 &    - \\
  PPR	  & 0.01  &  0.01 &   0.01 &  0.02  & 0.05 \\
  MARS    & 0.01  &  0.03 &   0.05 &  0.10  & 0.25 \\
  Tree    & 0.01  &  0.01 &   0.01 &  0.03  & 0.03 \\
  CAP     & 0.48  &  1.24 &   1.90 &  3.02  & 6.69 \\
  MCR     & 38.21 &496.54 &      - &      - & -   \\\hline
  \multicolumn{6}{c}{Problem 5} \\ \hline
  Method & $n=200$ & $n=500$ & $n=1000$ & $n=2000$ & $n=5000$ \\ \hline
  SCAIE   &  3.78 & 8.76 & 20.32 & 62.68 & 203.20 \\
  PPR	  &  0.01 & 0.02 &  0.03 &  0.05 &   0.12 \\
  MARS    &  0.01 & 0.01 &  0.02 &  0.03 &   0.04 \\
  Tree    &  0.01 & 0.01 &  0.01 &  0.02 &   0.03 \\\hline
  \end{tabular}
  \caption{\small Average running times (in seconds) of different methods for the shape-constrained additive index models (Problems~4 and~5).}
  \label{Tab:simtimeindex}
\end{table}


\begin{thebibliography}{99}
\bibitem[{Amari \emph{et~al.}(1996)}]{ACY1996}
Amari, S.,  Cichocki, A. and Yang, H. (1996) A new learning algorithm for blind signal separation.
In \emph{Advances in Neural Information Processing Systems}, 757--763 , MIT Press, Cambridge.

\bibitem[{Barlow \emph{et~al.}(1972)}]{BBBB1972}
Barlow, R.E., Bartholomew, D.J., Bremner, J.M. and Brunk, H.D. (1972) \emph{Statistical Inference under Order Restrictions}. 
\newblock Wiley, New York.

\bibitem[{Bhattacharya and Rao(1976)}]{BhattacharyaRao1976}
Bhattacharya, R.N. and Rao R.R. (1976) \emph{Normal Approximation and Asymptotic Expansions}.
\newblock Wiley, New York.

\bibitem[{Breiman \emph{et~al.}(1984)}]{BFOS1984}
Breiman, L., Friedman, J. H., Olshen, R. A, and Stone, C. J. (1984) \emph{Classification and Regression Trees}.
\newblock Wadsworth, California.

\bibitem[{Brunk (1958)}]{Brunk1958}
Brunk, H. D. (1958) On the estimation of parameters restricted by inequalities.
\newblock \emph{Annals of Mathematical Statistics}, \textbf{29}, 437--454. 

\bibitem[{Brunk(1970)}]{Brunk1970}
Brunk, H. D. (1970) Estimation of Isotonic Regression.
\newblock In \emph{Nonparametric Techniques in Statistical Inference}, 177--195, Cambridge University Press, Cambridge. 

\bibitem[{Chen and Samworth(2014)}]{ChenSamworth2014}
Chen, Y. and Samworth, R. J. (2014) scar: Shape-Constrained Additive Regression: a Maximum Likelihood Approach.
\newblock \texttt{R} package version 0.2-0,
\newblock \texttt{http://cran.r-project.org/web/packages/scar/}.

\bibitem[{D\"umbgen and Rufibach(2011)}]{DumbgenRufibach2011} 
D\"umbgen, L. and Rufibach, K. (2011) logcondens: computations related to univariate log-concave density estimation.
\newblock \emph{Journal of Statistical Software}, \textbf{39}, 1--28.

\bibitem[{D\"umbgen, Samworth and Schuhmacher(2011)}]{DSS2011}
D\"umbgen, L., Samworth, R. and Schuhmacher, D. (2011) Approximation by log-concave distributions with applications to regression. 
\newblock \emph{The Annals of Statistics}, \textbf{39}, 702--730.

\bibitem[{D\"umbgen, Samworth and Schuhmacher(2013)}]{DSS2013}
D\"umbgen, L., Samworth, R. J. and Schuhmacher, D. (2013) Stochastic search for semiparametric linear regression models.
\newblock In \emph{From Probability to Statistics and Back: High-Dimensional Models and Processes -- A Festschrift in Honor of Jon A. Wellner.}, 78--90.

\bibitem[{Friedman and Stuetzle(1981)}]{FriedmanStuetzle1981}
Friedman, J.H. and Stuetzle, W. (1981) Projection Pursuit Regression. 
\newblock \emph{Journal of the American Statistical Association}, \textbf{76}, 817--823.

\bibitem[{Friedman(1991)}]{Friedman1991}
Friedman, J.H. (1991) Multivariate adaptive regression splines. 
\newblock \emph{The Annals of Statistics}, \textbf{19}, 1--67.

\bibitem[{Gerschgorin(1931)}]{Gerschgorin1931}
Gerschgorin, S. (1931) \"Uber die Abgrenzung der Eigenwerte einer Matrix.
\newblock \emph{Izv. Akad. Nauk. USSR Otd. Fiz.-Mat. Nauk}, \textbf{6}, 749--754.

\bibitem[{Gradshteyn and Ryzhik(2007)}]{GradshteynRyzhik2007}
Gradshteyn, I. S. and Ryzhik, I. M. (2007) \emph{Table of Integrals, Series, and Products}.
\newblock Academic Press, San Diego, California.

\bibitem[{Groeneboom, Jongbloed and Wellner(2001)}]{GJW2001}
Groeneboom, P., Jongbloed, G. and Wellner, J.~A. (2001) Estimation of a convex
  function: Characterizations and asymptotic theory.
\newblock \emph{The Annals of Statistics}, \textbf{29}, 1653--1698.

\bibitem[{Groeneboom, Jongbloed and Wellner(2008)}]{GJW2008}
Groeneboom, P., Jongbloed, G. and Wellner, J. A. (2008) 
The support reduction algorithm for computing nonparametric function estimates in mixture models. 
\newblock \emph{Scandinavian Journal of Statistics}, \textbf{35}, 385--399.

\bibitem[{Guntuboyina and Sen(2013)}]{GuntuboyinaSen2013}
Guntuboyina, A. and Sen, B. (2013) Global risk bounds and adaptation in univariate convex regression.
\newblock Available at \texttt{http://arxiv.org/abs/1305.1648}.

\bibitem[{Hannah and Dunson(2013)}]{HannahDunson2013}
Hannah, L. A., and Dunson, D. B. (2013) Multivariate convex regression with adaptive partitioning. 
\newblock \emph{Journal of Machine Learning Research}, \textbf{14}, 3207--3240.

\bibitem[{Hanson and Pledger(1976)}]{HansonPledger1976}
Hanson, D.L. and Pledger, G. (1976) Consistency in concave regression.
\newblock \emph{The Annals of Statistics}, \textbf{4}, 1038--1050.

\bibitem[{Hastie and Tibshirani}(1986)]{HastieTibshirani1986}
Hastie, T., and Tibshirani, R. (1986) Generalized Additive Models (with discussion).
\newblock \emph{Statistical Science}, \textbf{1}, 297--318.

\bibitem[{Hastie and Tibshirani}(1990)]{HastieTibshirani1990}
Hastie, T., and  Tibshirani, R. (1990) \emph{Generalized Additive Models}. 
\newblock Chapman \& Hall, London.

\bibitem[{Hastie \emph{et~al.}}(2011)]{HTLHR2011}
Hastie, T., Tibshirani, R., Leisch, F., Hornik, K. and Ripley, B. D. (2011)
\texttt{mda}: mixture and flexible discriminant analysis.
\newblock \texttt{R} package version 0.4-2, \texttt{http://cran.r-project.org/web/packages/mda/}.

\bibitem[{Hayfield and Racine(2013)}]{HayfieldRacine2013}
Hayfield, T., and Racine, J. S. (2013)
\texttt{np}: Nonparametric kernel smoothing methods for mixed data types.
\newblock \texttt{R} package version 0.50-1, \texttt{http://cran.r-project.org/web/packages/np/}.

\bibitem[{Ichimura(1993)}]{Ichimura1993}
Ichimura, H. (1993) Semiparametric least squares (SLS) and weighted SLS estimation of single-index models.
\newblock \emph{Journal of Econometrics}, \textbf{58}, 71--120.

\bibitem[{J\o rgensen(1983)}]{Jorgensen1983}
J\o rgensen, B. (1983) Maximum likelihood estimation and large-sample inference for generalized linear and nonlinear regression models.
\newblock \emph{Biometrika}, \textbf{70}, 19--28.

\bibitem[{Kim and Samworth(2014)}]{KimSamworth2014}
Kim, A. K. H. and Samworth, R. J. (2014) Global rates of convergence in log-concave density estimation.
\newblock Available at \texttt{http://arxiv.org/abs/1404.2298}.

\bibitem[{Kleiber and  Zeileis(2013)}]{KleiberZeileis2013}
Kleiber, C. and  Zeileis, A. (2013)
\texttt{AER}: Applied Econometrics with R.
\newblock \texttt{R} package version 1.2-0, \texttt{http://cran.r-project.org/web/packages/AER/}.

\bibitem[{Long(1990)}]{Long1990}
Long, J.S. (1990) The Origins of Sex Differences in Science.
\newblock \emph{Social Forces}, \textbf{68}, 1297-1316.

\bibitem[{Long(1997)}]{Long1997}
Long, J.S. (1997). \emph{Regression Models for Categorical and Limited Dependent Variables}. 
\newblock Thousand Oaks: Sage Publications.

\bibitem[{Li and Racine(2007)}]{LiRacine2007}
Li, Q. and Racine, J.S. (2007) 
\emph{Nonparametric econometrics: theory and practice}. 
\newblock Princeton University Press, New Jersey. 

\bibitem[{Lim and Glynn(2012)}]{LimGlynn2012}
Lim, E., and Glynn, P. W. (2012) Consistency of multidimensional convex regression. 
\newblock \emph{Operations Research}, \textbf{60}, 196--208.

\bibitem[{Mammen and Yu(2007)}]{MammenYu2007}
Mammen, E. and Yu, K. (2007) Additive isotone regression.
\newblock In \emph{IMS Lecture Notes - Monograph series, Asymptotics: particles, processes and inverse problems}, \textbf{55}, 179--195.

\bibitem[{Meyer(2013a)}]{Meyer2013a}
Meyer, M.C. (2013a).  Semi-parametric Additive Constrained Regression.
\newblock \emph{Journal of Nonparametric Statistics}, \textbf{25}, 715--743.

\bibitem[{Meyer(2013b)}]{Meyer2013b}
Meyer, M.C. (2013b) A simple new algorithm for quadratic programming with applications in statistics.
\newblock \emph{Communications in Statistics}, \textbf{42}, 1126--1139.

\bibitem[{Nocedal and Wright(2006)}]{NocedalWright2006}
Nocedal, J., and Wright, S. J. (2006) \emph{Numerical Optimization}, 2nd edition.
\newblock Springer, New York.

\bibitem[{Price, Storn and Lampinen(2005)}]{PSL2005}
Price, K., Storn, R. and Lampinen, J. (2005) \emph{Differential evolution: A practical approach to global optimization}.
\newblock Springer-Verlag, Berlin.

\bibitem[{Pya(2012)}]{Pya2012}
Pya, N. (2012)
\texttt{scam}: Shape constrained additive models.
\newblock \texttt{R} package version 1.1-5, \texttt{http://cran.r-project.org/web/packages/scam/}.

\bibitem[{Pya and Wood(2014)}]{PyaWood2014}
Pya, N. and Wood, S. N. (2014) Shape constrained additive models.
\newblock \emph{Statistics and Computing}, to appear.  \texttt{DOI 10.1007/s11222-013-9448-7}.

\bibitem[{Ripley}(2012)]{Ripley2012}
Ripley, B. D. (2012)
\texttt{tree}: classification and regression trees.
\newblock \texttt{R} package version 1.0-33, \texttt{http://cran.r-project.org/web/packages/tree/}


\bibitem[{Samworth and Yuan(2012)}]{SamworthYuan2012}
Samworth, R.J. and Yuan, M. (2012) Independent component analysis via nonparametric
maximum likelihood estimation. 
\newblock \emph{The Annals of Statistics}, \textbf{40}, 2973--3002.

\bibitem[{Seijo and Sen(2011)}]{SeijoSen2011}
Seijo, E., and Sen, B. (2011) Nonparametric least squares estimation of a multivariate convex regression function. 
\newblock \emph{The Annals of Statistics}, \textbf{39}, 1633--1657.

\bibitem[{Stone(1986)}]{Stone1986}
Stone, C. (1986) The dimensionality reduction principle for generalized additive models.
\newblock \emph{The Annals of Statistics}, \textbf{14}, 590--606.

\bibitem[{van~der~Vaart and Wellner(1996)}]{vanderVaartWellner1996}
van~der~Vaart, A., and Wellner, J. A. (1996) \emph{Weak Convergence and Empirical Processes}.
\newblock Springer, New York.

\bibitem[{van~der~Vaart and Wellner(2000)}]{vanderVaartWellner2000}
van~der~Vaart, A., and Wellner, J. A. (2000) Preservation theorems for Glivenko--Cantelli and uniform Glivenko--Cantelli classes.
In \newblock \emph{High Dimensional Probability II}, 115--133, Birkhauser, Boston. 

\bibitem[{Wood(2004)}]{Wood2004}
Wood, S.N. (2004) Stable and efficient multiple smoothing parameter estimation for generalized
additive models. 
\newblock \emph{Journal of American Statistical Association}, \textbf{99}, 673--686.

\bibitem[{Wood(2006)}]{Wood2006}
Wood, S.N. (2006) \emph{Generalized Additive Models: An Introduction with R}. 
Chapman \& Hall, London.

\bibitem[{Wood(2008)}]{Wood2008}
Wood, S.N. (2008) Fast stable direct fitting and smoothness selection for generalized additive models.
\newblock \emph{Journal of the Royal Statistical Society, Series B}, \textbf{70}, 495--518.

\bibitem[{Wood(2012)}]{Wood2012}
Wood, S.N. (2012)
\texttt{mgcv}: Mixed GAM Computation Vehicle with GCV/AIC/REML smoothness estimation.
\newblock \texttt{R} package version 1.7-22, \texttt{http://cran.r-project.org/web/packages/mgcv/}.

\bibitem[{Yuan(2011)}]{Yuan2011}
Yuan, M. (2011) On the identifiability of additive index models. 
\newblock \emph{Statistica Sinica}, \textbf{21}, 1901--1911.
\end{thebibliography}
\end{document}